\documentclass[a4paper,11pt]{article}
\usepackage[T1]{fontenc}
\usepackage[utf8]{inputenc}
\usepackage[english]{babel}
\usepackage{lmodern}
\usepackage{amsmath,amsthm,amssymb,amsfonts}
\usepackage{mathtools}
\usepackage{color}
\usepackage{enumerate}
\usepackage{graphicx}
\usepackage{geometry}
\usepackage{subcaption}
\usepackage[overload]{empheq}
\usepackage{framed}
\usepackage{stmaryrd}
\usepackage[ruled,vlined]{algorithm2e}
\usepackage{hyperref}
\hypersetup{
        colorlinks=true,   
        linkcolor=blue,    
        citecolor=blue,    
        urlcolor=blue      
}

\newtheorem{theorem}{Theorem}[section]
\newtheorem{lemma}[theorem]{Lemma}
\newtheorem{proposition}[theorem]{Proposition}
\newtheorem{corollary}[theorem]{Corollary}
\newtheorem{definition}[theorem]{Definition}

\theoremstyle{remark}
\newtheorem{remark}[theorem]{Remark}
\newtheorem{example}[theorem]{Example}

\newcommand{\n}[1]{\|#1 \|}

\newcommand{\z}{\bar z}
\newcommand{\bx}{\mathbf x}
\newcommand{\bw}{\mathbf w}
\newcommand{\by}{\mathbf y}
\newcommand{\bz}{\mathbf z}
\newcommand{\ba}{\mathbf a}
\newcommand{\bv}{\mathbf v}
\newcommand{\bA}{\mathbf A}
\newcommand{\bmu}{\boldsymbol \mu}
\newcommand{\R}{\mathbb R}

\newcommand{\lr}[1]{\langle #1\rangle}

\DeclareMathOperator{\prox}{prox}

\DeclareMathOperator{\gra}{gra}

\DeclareMathOperator{\zer}{zer}
\DeclareMathOperator{\Fix}{Fix}

\DeclareMathOperator{\vspan}{span}
\DeclareMathOperator{\rank}{rank}

\newcommand{\Hilbert}{\mathcal{H}}
\newcommand{\setto}{\rightrightarrows}
\newcommand{\wto}{\rightharpoonup}

\DeclareMathOperator{\Id}{Id}
\DeclareMathOperator*{\argmin}{arg\,min}
\newcommand{\integ}[2]{\llbracket{#1},{#2}\rrbracket}

\title{Resolvent Splitting for Sums of Monotone Operators\\ with Minimal Lifting}

\author{Yura Malitsky\thanks{Department of Mathematics,
                             Link\"oping University,
                             581~83 Linköping, Sweden.
                             Email:~\href{href:yurii.malitskyi@liu.se}
                                         {yurii.malitskyi@liu.se}}
        \and
        Matthew K. Tam\thanks{School of Mathematics \& Statistics,
                             The University of Melbourne,
                             Parkville VIC 3010, Australia.
                             Email:~\href{href:matthew.tam@unimelb.edu.au}
                                         {matthew.tam@unimelb.edu.au}}
}

\begin{document}
\maketitle

\begin{abstract}
In this work, we study fixed point algorithms for finding a zero in the sum of $n\geq 2$ maximally monotone operators by using their resolvents. More precisely, we consider the class of such algorithms where each resolvent is evaluated only once per iteration. For any algorithm from this class, we show that the underlying fixed point operator is necessarily defined on a $d$-fold Cartesian product space with $d\geq n-1$. Further, we show that this bound is unimprovable by providing a family of examples for which $d=n-1$ is attained. This family includes the Douglas--Rachford algorithm as the special case when $n=2$. Applications of the new family of algorithms in distributed decentralised optimisation and multi-block extensions of the alternation direction method of multipliers (ADMM) are discussed.
\end{abstract}

\section{Introduction}
In this work, we study fixed point algorithms for finding a zero in the sum of finitely many maximally monotone operators defined on a real Hilbert space $\Hilbert$. That is, we consider problems of the form 
\begin{equation}\label{eq:mono inc intro}
 \text{find~}x\in\Hilbert\text{~such
   that~}0\in\sum_{i=1}^nA_i(x),
\end{equation}
where the set-valued operator $A_i\colon\Hilbert\setto\Hilbert$ is maximally
monotone for all $i\in\{1,\dots,n\}$. We focus on so-called \emph{backward schemes} for solving \eqref{eq:mono inc intro}. That is, we assume the operator $A_i$ is only available through its resolvent $J_{A_{i}}:=(\Id+A_{i})^{-1}$. Such methods are the backbone of modern optimisation~\cite{combettes2011pesquet}. Although resorting exclusively to the use of resolvents for \eqref{eq:mono inc intro} may seem restrictive, it is worth noting that algorithms have interpretations as the resolvents of appropriately choosen monotone operators~\cite{eckstein1992douglas,he2012convergence_pd}.

To be more precise, this work considers
the class of so-called \emph{frugal resolvent splittings} in the sense of
Ryu~\cite{ryu2020uniqueness} which form a special class of backward schemes. Roughly speaking, a fixed point algorithm for
solving \eqref{eq:mono inc intro} is a member of this class if it can be
described using only vector addition, scalar multiplication, and each of the
resolvents $J_{A_1},\dots,J_{A_n}$ once per iteration. The best known
examples of frugal resolvent splittings are the \emph{proximal point
  algorithm}~\cite{rockafellar1976monotone} for $n=1$ and the \emph{Douglas--Rachford algorithm}~\cite{lions1979splitting,eckstein1992douglas} for $n=2$.

The memory requirements of a fixed point algorithm can be measured using the
notion of \emph{lifting}~\cite{ryu2020uniqueness}. A fixed point
algorithm has \emph{$d$-fold lifting} if its underlying fixed point operator can
be defined on the $d$-fold Cartesian product
$\Hilbert^d:=\Hilbert\times\stackrel{(d)}{\dots}\times \Hilbert$. Since the
amount of memory needed to work with vectors in $\Hilbert^d$ is $d$-times larger than  vectors in $\Hilbert$, the quantity $d$ can be used to compare different algorithms. In practice, algorithms with less lifting (\emph{i.e.,}~smaller $d$ for given $n$) can be desirable as they allow for larger problems to be solved with less computational resources. The best algorithms according to this criteria are said to have \emph{minimal lifting}. That is, the smallest amount of lifting (for a given $n$) while still solving all feasible instance of \eqref{eq:mono inc intro}.

Motivated by this property, this work studies the relationship between the
number of monotone operators in~\eqref{eq:mono inc intro}, as specified by $n$,
and the minimal amount of lifting in frugal resolvent splitting for
solving~\eqref{eq:mono inc intro}, denoted by $d^*(n)$. The first few
values of $d^*(n)$ are known in the existing literature, as we now explain.
For $n=1$, the proximal point algorithm is a frugal resolvent splitting with
$1$-fold lifting (which is necessarily a minimal lifting), so $d^*(1)=1$.
For $n=2$, the Douglas--Rachford algorithm is a frugal resolvent
splitting with $1$-fold lifting, so $d^*(2)=1$. For $n=3$, Ryu
\cite{ryu2020uniqueness,aragon2020strengthened} devised a scheme with $2$-fold
lifting and established its minimality in \cite[Theorem~3]{ryu2020uniqueness}.
As a consequence, $d^*(3)=2$. For $n\geq 4$, the value of $d^*(n)$
has remained an
open problem. Indeed, Ryu's scheme for $n=3$ does not seem to generalise to the $n\geq 4$ setting (see Remark~\ref{r:extension to ryu}). Nevertheless, the pattern provided by the first few terms does suggest that $d^*(n)=n-1$ for $n\geq 2$.

In the first part of this work, we fully resolve the aforementioned open question by showing that it is indeed the case that $d^*(n)= n-1$ for $n\geq 2$. This is established in two steps: we first show that $d^*(n)\geq  n-1$ for $n\geq 2$ using techinques inspired by \cite{ryu2020uniqueness}, and we then show the bound cannot be improved by providing a new family of frugal resolvent splittings for \eqref{eq:mono inc intro} with $(n-1)$-fold lifting. We believe this family of algorithms, which does not rely on the usual product space reformulation (see Remark~\ref{ex:psf}), to be of interest in its own right. Indeed, in the second part of this work, we investigate implications and applications of the new family of algorithms for structured optimisation problems. More precisely, we use the family to devise two novel schemes: a method for distributed optimisation which is decentralised in the sense that it does not require a ``central coordinator'', and a multi-block extension of the \emph{alternating direction method of multipliers (ADMM).} Numerical examples are included to illustrate the methods, although this is not our main focus.

The remainder of this work is structured as follows. In Section~\ref{s:fpe}, we
recall the necessarily preliminaries on fixed point algorithms. In
Section~\ref{s:fpe lifting}, we establish the purported lower bound for $d^*(n)$ and, in Section~\ref{s:family}, we introduce the new family of frugal resolvent splittings and proof convergence. In Section~\ref{s:distributed}, we devise a scheme for distributed decentralised optimisation which uses resolvents and, in Section~\ref{s:admm}, we present our multi-block extension of ADMM.
Finally, Section~\ref{s:conclusion} concludes by outlining a number of directions and open questions for future research.

\section{Fixed Point Encodings}\label{s:fpe}
Throughout this work, $\Hilbert$ denotes a real Hilbert space with inner-product $\langle\cdot,\cdot\rangle$ and induced norm~$\|\cdot\|$. A set-valued operator $B\colon\Hilbert\setto\Hilbert$ is said to be \emph{monotone} if
 $$ \langle x-y,u-v\rangle\geq 0\qquad\forall (x,u),(y,v)\in\gra B:=\bigr\{(x,u):u\in B(x)\bigl\}. $$
A monotone operator is \emph{maximally monotone} if no proper extension is monotone. The \emph{resolvent} of an operator $B\colon\Hilbert\setto\Hilbert$ is the operator given by $J_B:=(I+B)^{-1}$. When $B$ is maximally monotone, its resolvent $J_B$ is single-valued with full domain, and firmly nonexpansive \cite{minty1962monotone}.

Let $\mathcal{A}_n$ denote the set of all $n$-tuples of maximally monotone operators on $\Hilbert$. In other words, $A=(A_1,\dots,A_n)\in\mathcal{A}_n$ if and only if $A_i\colon\Hilbert\setto\Hilbert$ is maximally monotone for all $i\in\{1,\dots,n\}$. Each $A\in\mathcal{A}_n$ induces an instance of the $n$-operator monotone inclusion problem given by
\begin{equation}\label{eq:mono inlc}
 \text{find~}x\in\Hilbert\text{~such that~}0\in\sum_{i=1}^nA_i(x).
\end{equation}
Note that the definition of $\mathcal{A}_n$ does not require the existence of a solution to \eqref{eq:mono inlc}. Further, due to commutativity of vector addition in $\Hilbert$, all cyclic permutations of an $n$-tuple $A\in\mathcal{A}_n$ induce the same instance of \eqref{eq:mono inlc}.

In the first part of this work, we study the structure of fixed point iterations for solving
\eqref{eq:mono inlc} whose update step
can be defined in terms of the resolvents of the monotone operators
$A_1,\dots,A_n$. Following Ryu~\cite{ryu2020uniqueness}, we formalise this
idea in terms of two operators both parameterised by $A\in\mathcal{A}_n$: a
\emph{fixed point operator}, denoted $T_A$, and a \emph{solution operator}, denoted $S_A$. The fixed point operator can be thought of as the basis for an iterative algorithm with the corresponding solution operator mapping its fixed points to solutions.

\begin{definition}[Fixed point encoding {\cite{ryu2020uniqueness}}]
A pair of operators $(T_A,S_A)$ is a \emph{fixed point encoding} for $\mathcal{A}_n$ if, for all $A\in\mathcal{A}_n$,
$$ \Fix T_A\neq\varnothing\iff\zer\left(\sum_{i=1}^nA_i\right)\neq\varnothing\text{~~and~~}
 \bz\in\Fix T_A  \implies  S_A(\bz)\in\zer\left(\sum_{i=1}^nA_i\right). $$
\end{definition}

\begin{example}\label{ex:lfpe}
The \emph{proximal point algorithm} is the fixed point encoding for $\mathcal{A}_1$ defined by $T_A=J_{A_1}$ and $S_A=\Id.$ The \emph{Douglas--Rachford algorithm} is the fixed point encoding for $\mathcal{A}_2$ defined by
\begin{equation}\label{eq:ex dr}
 T_A=\Id+J_{A_2}(2J_{A_1}-\Id)-J_{A_1}\text{~~and~~}S_A=J_{A_1}.
\end{equation}
Let $\bz=(z_1,z_2)\in\Hilbert^2$. Then \emph{Ryu's splitting algorithm} \cite{aragon2020strengthened,ryu2020uniqueness} is the fixed point encoding for $\mathcal{A}_3$ defined by
 $$ T_A(\bz) = \bz + \frac{1}{2}\begin{bmatrix}J_{A_3}\bigl(J_{A_1}(z_1)-z_1+J_{A_2}\bigl(J_{A_1}(z_1)+z_2\bigr)-z_2\bigr)-J_{A_1}(z_1)\\ J_{A_3}\bigl(J_{A_1}(z_1)-z_1+J_{A_2}\bigl(J_{A_1}(z_1)+z_2\bigr)-z_2\bigr)-J_{A_2}\bigl(J_{A_1}(z_1)+z_2\bigr)\end{bmatrix} $$
and $S_A(\bz)=J_{A_1}(z_1)$.
\end{example}

\begin{definition}[Resolvent splitting {\cite{ryu2020uniqueness}}]\label{def:resolvent splitting}
A fixed point encoding $(T_A,S_A)$ is a \emph{resolvent splitting} if, for all $A\in\mathcal{A}_n$, there is a finite procedure that evaluates $T_A$ and $S_A$ at a given point that uses only vector addition, scalar multiplication, and the resolvents of $A_1,\dots,A_n$.
\end{definition}
Definition~\ref{def:resolvent splitting} does not specify the number of times
per iteration that the resolvents of $A_1,\dots, A_n$ can be used. Thus,
without further restrictions, the computational cost per iteration of two
different resolvent splittings need not be the same. The following definition
provides a slight refinement as one way to address this.
\begin{definition}[Frugality {\cite{ryu2020uniqueness}}]\label{def:frugal}
A resolvent splitting $(T_A,S_A)$ is \emph{frugal} if, for all $A\in\mathcal{A}_n$, there is a finite procedure that evaluates $T_A$ and $S_A$ at a given point that uses only vector addition, scalar multiplication, and each of the resolvents of $A_1,\dots,A_n$ exactly once.
\end{definition}
It is worth emphasising that, given a point $z$, the procedure for evaluating
$T_A$ and $S_A$ described in Definition~\ref{def:frugal} must compute both
$T_A(z)$ and $S_A(z)$ using each of the resolvents precisely once. In other
words, the resolvent evaluations used to compute $S_A(z)$ must be the same as
those used to compute $T_A(z)$. As will be demonstrated in Section~\ref{s:fpe
  lifting}, every frugal resolvent splitting can be expressed and analysed in
terms of six coefficient matrices which fully define such a method.

\begin{example}
All fixed point encodings in Example~\ref{ex:lfpe} are frugal resolvent splittings.
For an example of a resolvent splitting for $A=(A_1)\in\mathcal{A}_1$ which is not frugal, consider $(T_A,S_A)$ given by
$$ T_A=\Id+J_{A_1}(2J_{A_1}-\Id)-J_{A_1}\text{~~and~~}S_A=J_{A_1}, $$
which coincides with the Douglas--Rachford algorithm applied to inclusion $0\in(A_1+A_1)(x)$. The resolvent splitting $(T_A,S_A)$ is not frugal in general because any procedure for computing $T_A$ requires the resolvent $J_{A_1}$ to be evaluated twice. 
\end{example}

\begin{definition}[Lifting {\cite{ryu2020uniqueness}}]
Let $d\in\mathbb{N}$. A fixed point encoding $(T_A,S_A)$ is a \emph{$d$-fold lifting} for $\mathcal{A}_n$ if $T_A\colon\Hilbert^d\to\Hilbert^d$ and $S_A\colon\Hilbert^d\to\Hilbert$.
\end{definition}
Using the well-known product space reformulation, it is straightforward to derive a frugal resolvent splitting for $\mathcal{A}_n$ with $n$-fold lifting. In what follows, denote the \emph{diagonal subspace} in $\Hilbert^n$ by
$$ \Delta_n := \bigl\{\bz=(z_1\dots,z_n)\in\Hilbert^n:z_1=\dots=z_n\bigr\}. $$
\begin{example}[Product space formulation]\label{ex:psf}
Let $A=(A_1,\dots,A_n)\in\mathcal{A}_n$. Then
 $$x\in\zer\left(\sum_{i=1}^nA_i\right) \iff \bx=(x,\dots,x)\in\zer\left(A+ N_{\Delta_n}\right), $$
where $N_{\Delta_n}$ denotes the (convex) normal cone to $\Delta_n$.
Consider the operators $T_A\colon\Hilbert^n\to\Hilbert^n$ and $S_A\colon\Hilbert^n\to\Hilbert$ given by
$$ T_A = \Id + J_A(2P_{\Delta_n}-\Id)-P_{\Delta_n}\text{~~and~~}S_A=P_1P_{\Delta_n}, $$
where $P_{\Delta_n}$ denotes the projection onto $\Delta_n$, $J_A=(J_{A_1},\dots,J_{A_n})$, and $P_1$ denotes the projection onto the first product coordinate of a vector in $\Hilbert^n$. Then $(T_A,S_A)$ is a frugal resolvent splitting of $\mathcal{A}_n$ with {$n$-fold} lifting. In fact, $T_A$ coincides with the Douglas--Rachford operator applied to the monotone operators $N_{\Delta_n}$ and $A$.
\end{example}

In light of Example~\ref{ex:psf}, it is natural to ask if there exists a frugal resolvent splitting for $\mathcal{A}_n$ with $d$-fold lifting for $d<n$. Indeed, this will be the main question we consider in the following section.

\section{Frugal Resolvent Splittings with Lifting}\label{s:fpe lifting}
Suppose $(T_A,S_A)$ is a frugal resolvent splitting for $\mathcal{A}_n$ with $d$-fold lifting where $d\leq n$. By definition, there exists a finite procedure for evaluating $T_A$ and $S_A$ using only vector addition, scalar multiplication and the resolvents of $A_1,\dots,A_n$ precisely once. The form of such a procedure may be completely described in terms of a number of matrices, as we now explain.

Consider evaluation of $T_A$ at an arbitrary point $\bz=(z_1,\dots,z_d)\in\Hilbert^d$. By frugality, each of the resolvents $J_{A_1},\dots,J_{A_n}$ is evaluated precisely once in the computation of $T_A(\bz)$. Thus there exist points $\bx=(x_1,\dots,x_n)\in\Hilbert^n$ and $\by=(y_1,\dots,y_n)\in\Hilbert^n$ used in the procedure such that
\begin{equation}\label{eq:can resolvent}
 \bx=J_{A}(\by) \iff 0 \in \bx-\by+A(\bx).
\end{equation}
The resolvents in \eqref{eq:can resolvent} are computed in some order in the
evaluation of $T_A(\bz)$. Without loss of generality, we assume they are
computed in the order $J_{A_1}(y_1),\dots,J_{A_n}(y_n)$. Since $T_A$ is a
resolvent splitting, $y_i$ is a linear combination of points already computed in
the process of evaluating $T_A(\bz)$. Thus, we must have
$ y_1\in\vspan\{z_1,\dots,z_d\}$ and
 $$ y_i\in\vspan\{z_1,\dots,z_d,x_1,\dots,x_{i-1},y_1,\dots,y_{i-1}\}=\vspan\{z_1,\dots,z_d,x_1,\dots,x_{i-1}\} \quad\forall i\geq 2, $$
with the coefficients in these linear combinations independent of $z$, $x$ and $y$. Thus, in a slight abuse of notation,\footnote{Strictly speaking, \eqref{eq:can y} should be written as $\by=(B\otimes\Id)\bz+(L\otimes\Id)\bx$ where $\otimes$ denotes the Kronecker product. In the special case when $\Hilbert=\mathbb{R}$, $B\otimes\Id=B$ and $L\otimes\Id=L$, and the two expressions coincide.} we may express this compactly as
\begin{equation}\label{eq:can y}
  \by=B\bz+L\bx.
\end{equation}
where $B\in\mathbb{R}^{n\times d}$ contains the coefficients of $z_1,\dots,z_n$ and $L\in\mathbb{R}^{n\times n}$ is a lower-triangular matrix with zeros on the diagonal which contains the coefficients of $x$. Since $T_A$ is a frugal resolvent splitting, we also have
 $$ T_A(\bz)\in\vspan\{z_1,\dots,z_d,x_1,\dots,x_n,y_1,\dots,y_n\}= \vspan\{z_1,\dots,z_d,x_1,\dots,x_n\}, $$
where, as before, the coefficients in the linear combinations independent of $\bz$ and $\bx$. Hence there exists $T_z\in\mathbb{R}^{d\times d}$ and $T_x\in\mathbb{R}^{d\times n}$ such that\footnote{We emphasise that the coefficient matrices $T_z$ and $T_x$ are constant and do not have any dependence on $\bz$ or $\bx$. The subscripts are merely labels to indicate which variables the coefficients belong to. An analogous remark applies to the coefficient matrices $S_z$ and $S_x$ in \eqref{eq:can SA}.}
\begin{equation}\label{eq:can TA}
  T_A(\bz) = T_z\bz+T_x\bx.
\end{equation}
Similarly, frugality also ensures
 $$ S_A(\bz)\in\vspan\{z_1,\dots,z_d,x_1,\dots,x_n,y_1,\dots,y_n\}= \vspan\{z_1,\dots,z_d,x_1,\dots,x_n\}, $$
hence there exists matrices $S_z\in\mathbb{R}^{1\times d}$ and $S_x\in\mathbb{R}^{1\times n}$ such that
\begin{equation}\label{eq:can SA}
 S_A(\bz) = S_z\bz+S_x\bx.
\end{equation}
Thus, altogether, any frugal resolvent splitting for $\mathcal{A}_n$ with $d$-fold lifting is completed described by \eqref{eq:can resolvent}, \eqref{eq:can y}, \eqref{eq:can TA} and \eqref{eq:can SA}, and the matrices $B,L,T_z,T_x,S_z,S_x$. With this in mind, the results in this section are formulated in terms of these matrices.

We shall require the following technical lemma.

\begin{lemma}\label{l:kerM}
Let $(T_A,S_A)$ be a frugal resolvent splitting for $\mathcal{A}_n$. Let $M$ denote the block matrix given by
$$ M := \begin{bmatrix}
 0       & \Id & -\Id & \Id \\
 B       &   L & -\Id &   0 \\
 T_z-\Id & T_x &    0 &   0 \\
\end{bmatrix}. $$
If $\bz\in\Fix T_A$, then there exists $\bv=\begin{bmatrix}\bz&\bx&\by&\ba\end{bmatrix}^\top\in\ker M$ with $\ba\in A(\bx)$. Conversely, if $\bv=\begin{bmatrix}\bz&\bx&\by&\ba\end{bmatrix}^\top\in\ker M$ and $\ba\in A(\bx)$, then $\bz\in \Fix T_A$, $\bx=J_A(\by)$ and $S_A(\bz)=S_z\bz+S_x\bx$.
\end{lemma}
\begin{proof}
Suppose $\bz\in\Fix T_A$. Then there exists $\bx,\by\in\Hilbert^n$ defined by
\eqref{eq:can resolvent}--\eqref{eq:can TA} such that
$$\begin{bmatrix} \bx-\by+A(\bx) \\ B\bz+L\bx-\by \\ (T_z-\Id)\bz+T_x\bx \end{bmatrix} = M\begin{bmatrix}\bz\\ \bx\\ \by \\ A(\bx)\\ \end{bmatrix} \ni 0.$$
Consequently, there exists $\ba\in A(\bx)$ such that $\bv:=\begin{bmatrix}\bz& \bx& \by& \ba \end{bmatrix}^\top\in\ker M$.

Conversely, suppose $\bv=\begin{bmatrix}\bz&\bx&\by&\ba\end{bmatrix}^\top\in\ker M$ and $\ba\in A(\bx)$. The first row of $M$ gives $0=\bx-\by+\ba\in \bx-\by+A(\bx) \implies \bx=J_A(\by)$, the second row of $M$ gives $\by=B\bz+L\bx$, and the last row of $M$ together with \eqref{eq:can TA} gives $\bz=T_z\bz+T_x\bx=T_A(\bz)\implies \bz\in\Fix T_A$.  
Thus, as $\bz,\bx,\by$ satisfy \eqref{eq:can resolvent}--\eqref{eq:can TA}, \eqref{eq:can SA} implies $S_A(\bz)=S_z\bz+S_x\bx$.
\end{proof}

The following result shows that the solution mapping of a frugal resolvent splitting is necessarily of a specific form.
\begin{proposition}[Solution mappings]\label{prop:soln map}
Let $(T_A,S_A)$ be a frugal resolvent splitting for $\mathcal{A}_n$. Then, for all $\bar{\bz}\in\Fix T_A$ and $\bar{\bx}=J_A(\bar{\by})$, we have
$$S_A(\bar{\bz})=\frac{1}{n}\sum_{i=1}^n\bar{y}_i=\bar{x}_1=\dots=\bar{x}_d. $$
\end{proposition}
\begin{proof}
Let $A\in\mathcal{A}_n$, let $\bar{\bz}\in\Fix T_A$ and set $x^*=S_A(\bar{\bz})$. By Lemma~\ref{l:kerM}, there exists a vector $\bv:=\begin{bmatrix}\bar{\bz}&\bar{\bx}&\bar{\by}&\bar{\ba}\end{bmatrix}^\top\in\ker M$ with $\bar{\ba}\in A(\bar{\bx})$ and $x^*=S_A(\bar{\bz})=S_z\bar{\bz}+S_x\bar{\bx}$.
Next, consider the $n$-tuples of operators
$A^{(0)},A^{(1)},\dots, A^{(n)}\in\mathcal{A}_n$ given by
 $$ A^{(0)}(\bx) = \bar{\ba}\text{~~and~~}A^{(j)}(\bx) = \bar{\ba} + \begin{bmatrix}0\\ \vdots \\ x_j-\bar{x}_j \\ \vdots \\ 0 \\ \end{bmatrix}\quad\forall j=1,\dots,n.$$
Since $\bv\in\ker M$ and $\bar{\ba}=A^{(j)}(\bar{\bx})$,
Lemma~\ref{l:kerM} implies $\bar{\bz}\in\Fix T_{A^{(j)}}$ and
$\bar{\bx}=J_{A^{(j)}}(\bar{\by})$ and
$S_{A^{(j)}}(\bar{\bz})=S_z\bar{\bz}+S_x\bar{\bx}=x^*$.
Consequently, we have $0 = \sum_{i=1}^nA^{(0)}_i(x^*) = \sum_{i=1}^n\bar{a}_i$ and
$$ 0 = \sum_{i=1}^nA^{(j)}_i(x^*) = \sum_{i=1}^n\bar{a}_i + x^*-\bar{x}_j =x^*-\bar{x}_j\quad\forall j=1,\dots,n, $$
from which it follows that $x^*=\bar{x}_1=\dots=\bar{x}_n$. Finally, since $\bar{\bx}=J_{A^{(0)}}(\bar{\by})$, we have $\bar{\by}-\bar{\bx}=A^{(0)}(\bar{\bx})=\bar{\ba}$ and hence
$ \sum_{i=1}^n\bar{y}_i-nx^* = \sum_{i=1}^n\bar{a}_i=0$. The proof is now complete.
\end{proof}

Thanks to Example~\ref{ex:psf}, we know that there exists a frugal resolvent splitting for $\mathcal{A}_n$ with $d$-fold lifting with $d=n$. The following propositions show that there are no frugal resolvent splittings with lifting when $d\leq n-2$, and thereby extends \cite[Theorem~3]{ryu2020uniqueness} beyond the setting with $n=3$.
\begin{theorem}[Minimal lifting]\label{th:lifting}
Let $(T_A,S_A)$ be a frugal resolvent splitting for $\mathcal{A}_n$ with $d$-fold lifting. If $n\geq 2$, then $d\geq n-1$.
\end{theorem}
\begin{proof}
Suppose, by way of a contradiction, that $(T_A,S_A)$ is a frugal resolvent
splitting for $\mathcal{A}_n$ with $d$-fold lifting such that $d\leq n-2$. Let
$A\in\mathcal{A}_n$ with $\zer(\sum_{i=1}^nA_i)\neq\varnothing$ and
$\bz\in\Fix T_A$. By Lemma~\ref{l:kerM}, there exists
$\bv=\begin{bmatrix}\bz&\bx&\by&\ba\end{bmatrix}^\top\in\ker M$ with
$\ba\in A(\bx)$. In particular, the last row of $M$ implies that
$ 0 = (T_z-\Id)\bz+T_x\bx$.
Since $T_x\in\mathbb{R}^{d\times n}$ and $d\leq n-2$, the rank-nullity theorem implies
$\dim\ker T_x = n - \rank T_x \geq n - d \geq 2.$
Thus, since $\dim\Delta_n=1$, there exists $\bar{\bx}\not\in\Delta_n$ such that $T_x\bx=T_x\bar{\bx}$.

Set $\bar{\bz}:=\bz$, $\bar{\by}:=B\bar{\bz}+L\bar{\bx}$, $\bar{\ba}:=\bar{\by}-\bar{\bx}$ and consider $\bar{A}\in\mathcal{A}_n$ given by $\bar{A}(\bx)=\bar{\ba}$. Then $\bar{\bv}=\begin{bmatrix}\bar{\bz}&\bar{\bx}&\bar{\by}&\bar{\ba}\end{bmatrix}^\top\in\ker M$ with $\bar{\ba}\in\bar{A}(\bar{\bx})$. By the Lemma~\ref{l:kerM}, $\bar{\bz}\in\Fix T_{\bar{A}}$ and $\bar{\bx}=J_{\bar{A}}(\bar{\by})$. On the other hand, Proposition~\ref{prop:soln map} implies $\bar{\bx}\in\Delta_n$. Thus a contradiction is obtained and so completes the proof.
\end{proof}

Theorem~\ref{th:lifting} provides a lower bound on the value of $d$ for frugal resolvent splittings with $d$-folding lifting, but says nothing about whether this bound is tight. In the following section, we address this question by showing that it is indeed unimprovable for $\mathcal{A}_n$.

\begin{remark}
The definition of a resolvent splitting (Definition~\ref{def:resolvent splitting}) only allows the resolvents $J_{A_1},\dots,J_{A_n}$ to be used in evaluation of $T_A$ and $S_A$.
Since the resolvent $J_{\omega_iA_i}$ for parameter $\omega_i>0$ is typically no more difficult to compute than the resolvent $J_{A_i}$, a natural generalisation of Definition~\ref{def:resolvent splitting} is to allow the resolvents $J_{\omega_1A_1},\dots,J_{\omega_nA_n}$ for parameters $\omega_1,\dots,\omega_n>0$, not necessarily all equal. After the publication of this work, this setting has been considered in \cite{aragon2022primal} under the name ``paramterised resolvent splittings''. In the same work, it was shown that the conclusion of Theorem~\ref{th:lifting} also holds for parameterised resolvents splitting \cite[Theorem~2.10]{aragon2022primal}. Examples of frugal parameterised resolvents splitting which attain this bound can be found in \cite{condat2021proximal,campoy2022product}.
\end{remark}

\section{A Family of Resolvent Splitting}\label{s:family}
In this section, we provide  an example of a frugal resolvent splitting for $\mathcal{A}_n$ with $(n-1)$-fold lifting, assuming $n\geq 2$. Due to Theorem~\ref{th:lifting} this is unimprovable for this problem class. In what follows, the set of integers between $k,l\in\mathbb{N}$ is denoted
 $$\integ{k}{l} := \begin{cases}
                     \{k,k+1,\dots,l\} & \text{if~}k\leq l, \\
                     \varnothing & \text{otherwise.}
                   \end{cases}$$
Let $\gamma\in(0,1)$ and $A=(A_1,\dots,A_n)\in\mathcal{A}_n$. Consider the operator $T_A\colon\Hilbert^{n-1}\to\Hilbert^{n-1}$ given by
\begin{equation}\label{eq:T_A}
T_A(\bz)
:=  \bz +
\gamma\begin{pmatrix}
x_2-x_1 \\
x_3-x_2 \\
\vdots \\
x_{n}-x_{n-1} \\
\end{pmatrix}
\end{equation}
where $\bx=(x_1,\dots,x_n)\in\Hilbert^n$ depends on $\bz=(z_1,\dots,z_{n-1})\in\Hilbert^{n-1}$ and is given by
\begin{equation}\label{eq:xi} \begin{cases}
    x_1=J_{A_1}(z_1),  \\
    x_i=J_{A_i}(z_i+x_{i-1}-z_{i-1}) \quad \forall i\in\integ{2}{n-1}, \\
    x_n=J_{A_n}(x_1+x_{n-1}-z_{n-1}). 
   \end{cases}
\end{equation}

\begin{remark}\label{r:special cases}
For $n=2$, the operator $T_A$ described by \eqref{eq:T_A}--\eqref{eq:xi} is an under-relaxation of the standard Douglas--Rachford algorithm. Indeed, in this case, $T_A\colon\Hilbert\to\Hilbert$ can be expressed as
\begin{equation}\label{eq:dr gamma}
 T_A(z) = z + \gamma(x_2-x_1)= z+\gamma\left(J_{A_2}(2J_{A_1}(z)-z) - J_{A_1}(z)\right) = \left(1-\frac{\gamma}{2}\right)\Id + \frac{\gamma}{2}R_{A_2}R_{A_1},
\end{equation}
where $R_{A_i}=2J_{A_i}-\Id$ denotes the \emph{reflected resolvent} of the
monotone operator $A_i$. The (unrelaxed) Douglas--Rachford algorithm in
\eqref{eq:ex dr} corresponds to the limiting case with $\gamma=1$. For $n=3$,
the operator $T_A$ is different than the one described by Ryu in
\cite[Theorem~4]{ryu2020uniqueness} (see also \cite[Theorem
8]{aragon2020strengthened}), which in the notation used in this section, corresponds to
\begin{equation}\label{eq:ryu our notation}
\bz \mapsto  \bz +
\gamma\begin{pmatrix}
x_3-x_1 \\
x_3-x_2 \\
\end{pmatrix}\text{~~where~~}\left\{\begin{aligned}
                              x_1 &= J_{A_1}(z_1) \\
                              x_2 &= J_{A_2}(z_2+x_1) \\
                              x_3 &= J_{A_3}(x_1-z_1+x_2-z_2).
                             \end{aligned}\right.
\end{equation}
As we will discuss in Remark~\ref{r:extension to ryu}, it is currently not obvious how to extend this method to the setting where $n\geq 4$.
\end{remark}

Let $\Omega$ denote the set of all points $(\bar{\bz},\bar{x})\in\Hilbert^{n-1}\times\Hilbert
$ where $\bar{\bz}=(\z_1,\dots,\z_{n-1})$ satisfying
\begin{equation*}
\begin{cases}
\bar{x}=J_{A_{1}}(\bar{z}_1), \\
\bar{x}=J_{A_i}(\bar{z}_i-\bar{z}_{i-1}+\bar{x})\quad  i\in\integ{2}{n-1} ,\\
\bar{x} =J_{A_n}(2\bar{x}-z_{n-1}).
\end{cases}
\end{equation*}

\begin{lemma}\label{l:fixed points}
Let $n\geq 2$, $A=(A_1,\dots,A_n)\in\mathcal{A}_n$ and $\gamma>0$. The following assertions hold.
\begin{enumerate}[(a)]
\item\label{l:fixed points a} If $\bar{\bz}\in\Fix T_A$, then there exists $\bar{x}\in\Hilbert$ such that $(\bar{\bz},\bar{x})\in\Omega$.
\item\label{l:fixed points b} If $\bar{x}\in \zer\left(\sum_{i=1}^nA_i\right)$, then there exists $\bar{\bz}\in\Hilbert$ such that $(\bar{\bz},\bar{x})\in\Omega$.
\item\label{l:fixed points c} If $(\bar{\bz},\bar{x})\in\Omega$, then $\bar{\bz}\in\Fix T_A$ and $\bar{x}\in \zer\left(\sum_{i=1}^nA_i\right)$.
\end{enumerate}
Consequently,
$$ \Fix T_A\neq\varnothing \iff \Omega\neq\varnothing \iff \zer\left(\sum_{i=1}^nA_i\right)\neq\varnothing.$$
\end{lemma}
\begin{proof}
\eqref{l:fixed points a}:~Let $\bar{\bz}\in\Fix T_A$ and set $\bar{x}:=J_{A_1}(\bar{z}_1)$. Since $T_A(\bar{\bz})=\bar{\bz}$, \eqref{eq:T_A} implies $(\bar{\bz},\bar{x})\in\Omega$.

\eqref{l:fixed points b}:~Let $\bar{x}\in \zer\left(\sum_{i=1}^nA_i\right)$. Then there exists $\mathbf{v}=(v_1,\dots,v_n)\in\Hilbert^n$ such that $_i\in A_i(\bar{x})$ and $\sum_{i=1}^nv_i=0$. Define the vector $\bar{\bz}=(z_1,\dots,z_{n-1})\in\Hilbert^{n-1}$ according to
$$ \begin{cases}
\bar{z}_1 := \bar{x}+v_1 \in (\Id+A_1)\bar{x}, \\
\bar{z}_i := v_{i} + \bar{z}_{i-1} =(\bar{x}+v_i)-\bar{x}+\bar{z}_{i-1} \in (\Id+A_i)(\bar{x}) - \bar{x}+\bar{z}_{i-1}\quad \forall i\in\integ{2}{n-1}.
   \end{cases}$$
Then $\bar{x}=J_{A_1}(z_1)$ and $\bar{x}=J_{A_i}(\bar{z}_i-\bar{z}_{i-1}+\bar{x})$ for $i\in\integ{2}{n-1}$. Furthermore, we have
$$ (\Id+A_n)(\bar{x})\ni \bar{x}+v_n = \bar{x}-\sum_{i=1}^{n-1}v_i = \bar{x} -(\bar{z}_1-\bar{x})-\sum_{i=2}^{n-1}(\bar{z}_i-\bar{z}_{i-1}) = 2\bar{x}-\bar{z}_{n-1},$$
which implies that $\bar{x}=J_{A_n}(2\bar{x}-\bar{z}_{n-1})$.
Altogether, we have $(\bar{\bz},\bar{x})\in\Omega$.

\eqref{l:fixed points c}:~Let $(\bar{\bz},\bar{x})\in\Omega$. It follows that $\bar{\bz}\in\Fix T$. Further, the definition of the resolvent implies
$$ \begin{cases}
    A_1(\bar{x})\ni \bar{z}_1-\bar{x} \\
    A_i(\bar{x})\ni\bar{z}_i-\bar{z}_{i-1} \quad \forall i\in\integ{2}{n-1}  \\
    A_n(\bar{x})\ni\bar{x}-z_{n-1}.
   \end{cases} $$
Summing together the above inclusions gives $\bar{x}\in \zer\left(\sum_{i=1}^nA_i\right)$, which completes the proof.
\end{proof}

\begin{lemma}\label{l:T_A ne}
Let $n\geq 2$, $A=(A_1,\dots,A_n)\in\mathcal{A}$ and $\gamma>0$. Then, for all $\bz=(z_1,\dots,z_n)\in\Hilbert^{n-1}$ and $\bar{\bz}=(\bar{z}_1,\dots,\bar{z}_n)\in\Hilbert^{n-1}$, we have
\begin{multline}\label{eq:T_A ne}
 \|T_A(\bz)-T_A(\bar{\bz})\|^2 + \frac{1-\gamma}{\gamma}\|(\Id-T_A)(\bz)-(\Id-T_A)(\bar{\bz})\|^2 \\
 + \frac{1}{\gamma}\bigl\|\sum_{i=1}^{n-1}(\Id-T_A)(\bz)_i-\sum_{i=1}^{n-1}(\Id-T_A)(\bar{\bz})_i\bigr\|^2 \leq \|\bz-\bar{\bz}\|^2.
\end{multline}
In particular, if $\gamma\in(0,1)$, then $T_A$ is $\gamma$-averaged
nonexpansive.
\end{lemma}
\begin{proof}
For convenience, denote $\bz^+:=T_A(\bz)$ and $\bar{\bz}^+:=T_A(\bar{\bz})$. Further, let $\bx=(x_1,\dots,x_n)\in\Hilbert^n$ be given by \eqref{eq:xi} and let $\bar{\bx}=(\bar{x}_1,\dots,\bar{x}_n)\in\Hilbert^n$ be given analogously. Since $z_1-x_1\in A_1(x_1)$ and $\bar{z}_1-\bar{x}_1\in A_1(\bar{x}_1)$, monotonicity of $A_1$ implies
\begin{equation}\label{eq:mono A1}
\begin{aligned}
 0 &\leq \langle x_1-\bar{x}_1,(z_1-x_1)-(\bar{z}_1-\bar{x}_1)\rangle \\
   &= \langle x_2-\bar{x}_1,(z_1-x_1)-(\bar{z}_1-\bar{x}_1)\rangle + \langle x_1-x_2,(z_1-x_1)-(\bar{z}_1-\bar{x}_1)\rangle.
\end{aligned}
\end{equation}
For $i\in\integ{2}{n-1}$, $z_i-z_{i-1}+x_{i-1}-x_i\in A_i(x_i)$ and $\bar{z}_i-\bar{z}_{i-1}+\bar{x}_{i-1}-\bar{x}_i\in A_i(\bar{x}_i)$. Thus monotonicity of $A_i$ yields
\begin{align*}
 0 &\leq \langle x_i-\bar{x}_i,(z_i-z_{i-1}+x_{i-1}-x_i)-(\bar{z}_i-\bar{z}_{i-1}+\bar{x}_{i-1}-\bar{x}_i)\rangle \\
   &= \langle x_i-\bar{x}_i,(z_i-x_i)-(\bar{z}_i-\bar{x}_i)\rangle  - \langle x_i-\bar{x}_i,(z_{i-1}-x_{i-1})-(\bar{z}_{i-1}-\bar{x}_{i-1})\rangle\\
   &= \langle x_{i+1}-\bar{x}_i,(z_i-x_i)-(\bar{z}_i-\bar{x}_i)\rangle + \langle x_i-x_{i+1},(z_i-x_i)-(\bar{z}_i-\bar{x}_i)\rangle \\
   &\hspace{1cm} - \langle x_i-\bar{x}_{i-1},(z_{i-1}-x_{i-1})-(\bar{z}_{i-1}-\bar{x}_{i-1})\rangle-\langle \bar{x}_{i-1}-\bar{x}_i,(z_{i-1}-x_{i-1})-(\bar{z}_{i-1}-\bar{x}_{i-1})\rangle.
\end{align*}
Summing this inequality for $i\in\integ{2}{n-1}$ and simplifying gives
\begin{multline}\label{eq:mono Ai}
0 \leq  \langle x_{n}-\bar{x}_{n-1},(z_{n-1}-x_{n-1})-(\bar{z}_{n-1}-\bar{x}_{n-1}\rangle - \langle x_2-\bar{x}_{1},(z_{1}-x_{1})-(\bar{z}_{1}-\bar{x}_{1})\rangle \\
   +\sum_{i=2}^{n-1}\langle x_i-x_{i+1},(z_i-x_i)-(\bar{z}_i-\bar{x}_i)\rangle-\sum_{i=1}^{n-2}\langle \bar{x}_{i}-\bar{x}_{i+1},(z_{i}-x_{i})-(\bar{z}_{i}-\bar{x}_{i})\rangle.
\end{multline}
Since $x_1+x_{n-1}-x_n-z_{n-1}\in A_n(x_n)$ and $\bar{x}_1+\bar{x}_{n-1}-\bar{x}_n-\bar{z}_{n-1}\in A_n(\bar{x}_n)$, monotonicity of $A_n$ gives
\begin{equation}\label{eq:mono An}
\begin{aligned}
 0 &\leq \langle x_n-\bar{x}_n,(x_1+x_{n-1}-x_n-z_{n-1})-(\bar{x}_1+\bar{x}_{n-1}-\bar{x}_n-\bar{z}_{n-1})\rangle  \\
   &=  \langle x_n-\bar{x}_n,(x_{n-1}-z_{n-1})-(\bar{x}_{n-1}-\bar{z}_{n-1})\rangle + \langle x_n-\bar{x}_n,(x_1-\bar{x}_1)-(x_n-\bar{x}_n)\rangle
        \\
   &= -\langle x_n-\bar{x}_{n-1},(z_{n-1}-x_{n-1})-(\bar{z}_{n-1}-\bar{x}_{n-1})\rangle \\
   &\qquad - \langle \bar{x}_n-\bar{x}_{n-1},(z_{n-1}-x_{n-1})-(\bar{z}_{n-1}-\bar{x}_{n-1})\rangle \\
   &\qquad\qquad  + \frac{1}{2}\left(\|x_1-\bar{x}_1\|^2-\|x_n-\bar{x}_n\|^2-\|(x_1-x_n)-(\bar{x}_1-\bar{x}_n)\|^2\right).
\end{aligned}
\end{equation}
Adding \eqref{eq:mono A1}, \eqref{eq:mono Ai} and \eqref{eq:mono An} and rearranging gives
\begin{multline}\label{eq:before key}
0\leq \sum_{i=1}^{n-1}\langle (x_i-\bar{x}_i)-(x_{i+1}-\bar{x}_{i+1}),\bar{x}_i-x_i\rangle-\sum_{i=1}^{n-1}\langle (x_i-\bar{x}_i)-(x_{i+1}-\bar{x}_{i+1}),z_i-\bar{z}_i\rangle \\
+ \frac{1}{2}\left(\|x_1-\bar{x}_1\|^2-\|x_n-\bar{x}_n\|^2-\|(x_1-x_n)-(\bar{x}_1-\bar{x}_n)\|^2\right).
\end{multline}
The first term in \eqref{eq:before key} can be expressed as
\begin{equation}\label{eq:ip x}
\begin{aligned}
&\sum_{i=1}^{n-1}\langle (x_i-\bar{x}_i)-(x_{i+1}-\bar{x}_{i+1}),\bar{x}_i-x_i\rangle\\
&= \frac{1}{2}\sum_{i=1}^{n-1}\left(\|x_{i+1}-\bar{x}_{i+1}\|^2-\|x_i-\bar{x}_i\|^2-\|(x_i-x_{i+1})-(\bar{x}_i-\bar{x}_{i+1})\|^2 \right) \\
&= \frac{1}{2}\left(\|x_n-\bar{x}_n\|^2-\|x_1-\bar{x}_1\|^2 - \frac{1}{\gamma^2}\sum_{i=1}^{n-1}\|(z_i-z_{i}^+)-(\bar{z}_i-\bar{z}^+)\|^2\right) \\
&= \frac{1}{2}\left(\|x_n-\bar{x}_n\|^2-\|x_1-\bar{x}_1\|^2 - \frac{1}{\gamma^2}\|(\bz-\bz^+)-(\bar{\bz}-\bar{\bz}^+)\|^2\right),
\end{aligned}
\end{equation}
and the second term in \eqref{eq:before key} can be expressed as
\begin{equation}\label{eq:ip xx}
\begin{aligned}
&\sum_{i=1}^{n-1}\langle (x_i-x_{i+1})-(\bar{x}_i-\bar{x}_{i+1}),z_i-\bar{z}_i\rangle \\
&= \frac{1}{\gamma}\sum_{i=1}^{n-1}\langle (z_i-z_i^+)-(\bar{z}_i-\bar{z}_i^+),z_i-\bar{z}_i\rangle \\
&= \frac{1}{\gamma}\langle (\bz-\bz^+)-(\bar{\bz}-\bar{\bz}^+),\bz-\bar{\bz}\rangle \\
&= \frac{1}{2\gamma}\left(\|(\bz-\bz^+)-(\bar{\bz}-\bar{\bz}^+)\|^2+\|\bz-\bar{\bz}\|^2-\|\bz^+-\bar{\bz}^+\|^2 \right).
\end{aligned}
\end{equation}
Thus, substituting \eqref{eq:ip x} and \eqref{eq:ip xx} into \eqref{eq:before key} and simplifying gives
\begin{multline}\label{eq:2nd last}
\|\bz^+-\bar{\bz}^+\|^2+\left(\frac{1}{\gamma}-1\right)\|(\bz-\bz^+)-(\bar{\bz}-\bar{\bz}^+)\|^2
+ \gamma\|(x_1-x_n)-(\bar{x}_1-\bar{x}_n)\|^2
 \leq \|\bz-\bar{\bz}\|^2.
\end{multline}
Note that \eqref{eq:xi} implies
$$\gamma(x_1-x_n)-\gamma(\bar{x}_1-\bar{x}_n) = \gamma\sum_{i=1}^{n-1}(x_i-x_{i+1})-\gamma\sum_{i=1}^n(\bar{x}_i-\bar{x}_{i+1}) = \sum_{i=1}^{n-1}(z_i-z_i^+)-\sum_{i=1}^{n-1}(\bar{z}_i-\bar{z}_i^+).$$
Substituting this  into \eqref{eq:2nd last} gives \eqref{eq:T_A ne}, which completes the proof.
\end{proof}
\begin{remark}
In general, Lemma~\ref{l:T_A ne} cannot be improved in the sense that
$T_A$ need not be averaged nonexpansive if $\gamma\geq 1$. Indeed, consider the setting with $n\geq 3$, $\gamma=1$, and $A_1=\dots=A_n=0$. Then $J_{A_1}=J_{A_2}=\dots=J_{A_n}=\Id$ and hence
$$ T_A(\bz) = T_A\begin{pmatrix}
z_1 \\
z_2 \\
\vdots \\
z_{n-2} \\
z_{n-1} \\
\end{pmatrix} = \begin{pmatrix}
z_2 \\
z_3 \\
\vdots \\
z_{n-1} \\
z_1 \\
\end{pmatrix}.$$
Consequently, $T_A$ is an isometry and hence not averaged nonexpansive.

However, in the special case when $n=2$, Lemma~\ref{l:T_A ne} recovers the known result that $T_A$ is averaged nonexpansive when $\gamma\in(0,2)$. Indeed, if $n=2$, then
$$ \bigl\|\sum_{i=1}^{n-1}(z_i-z_i^+)-\sum_{i=1}^{n-1}(\bar{z}_i-\bar{z}_i^+)\bigr\|^2 = \bigl\|(z_1-z_1^+)-(\bar{z}_1-\bar{z}_1^+)\bigr\|^2. $$
Consequently, the second and third terms in \eqref{eq:T_A ne} can be combined to give
$$ \|T_A(z)-T_A(\bar{z})\|^2 + \frac{2-\gamma}{\gamma}\|(\Id-T_A)(z)-(\Id-T_A)(\bar{z})\|^2 \leq \|z-\bar{z}\|^2. $$
Thus, if $n=2$ and $\gamma\in(0,2)$, then $T_A$ is $\frac{\gamma}{2}$-averaged nonexpansive.
\end{remark}
 
The following theorem, which is concerned with convergence of Algorithm~\ref{alg}, is our main result for this section.  
  
\begin{algorithm}[!ht]
\caption{Minimal resolvent splitting for finding a zero of $\sum_{i=1}^nA_i$ when $n\geq 2$. \label{alg}}
\SetKwInOut{Input}{Input}
\Input{Choose $\bz^0=(z^0_1,\dots,z^0_{n-1})\in\Hilbert^{n-1}$ and $\gamma\in(0,1)$.}
\smallskip
\For{$k=1,2,\dots$}{
Compute $\bz^{k+1}=(z_1^{k+1},\dots,z_{n-1}^{k+1})\in\Hilbert^{n-1}$ according to
\begin{equation}\label{eq:th T_A}
 \bz^{k+1} = T_A(\bz^k) = \bz^k  + \gamma \begin{pmatrix}
x_2^k-x_1^k \\
x_3^k-x_2^k \\
\vdots \\
x_{n}^k-x_{n-1}^k \\
\end{pmatrix},
\end{equation}
where $\bx^k=(x_1^k,\dots,x_n^k)\in \Hilbert^n$ is given by
\begin{equation}\label{eq:th J_A}
 \begin{cases}
x_1^k=J_{A_1}(z_1^k), \\
x_i^k=J_{A_i}(z_i^k-z_{i-1}^k+x_{i-1}^k) \qquad \forall i\in\integ{2}{n-1}, \\
x_n^k=J_{A_n}(x_1^k+x_{n-1}^k-z_{n-1}^k). \\
\end{cases}
\end{equation}
}
\end{algorithm}

\begin{theorem}\label{th:main}
Let $n\geq 2$, let $A=(A_1,\dots,A_n)\in\mathcal{A}_n$ with $\zer\left(\sum_{i=1}^nA_i\right)\neq\varnothing$, and let $\gamma\in(0,1)$. Given $\bz^0\in\Hilbert^{n-1}$, let $(\bz^k)\subseteq\Hilbert^{n-1}$ and $(\bx^k)\subseteq\Hilbert^n$ be the sequences given by \eqref{eq:th T_A} and \eqref{eq:th J_A}, respectively.
Then the following assertions hold.
\begin{enumerate}[(a)]
\item The sequence $(\bz^k)$ converges weakly to a point $\bz\in\Fix T_A$.
\item The sequence $(\bx^k)$ converges weakly to a point $(x,\dots,x)\in\Hilbert^n$ with $x\in\zer\left(\sum_{i=1}^nA_i\right)$.
\end{enumerate}
\end{theorem}
\begin{proof}
Since $\zer\left(\sum_{i=1}^nA_i\right)\neq\varnothing$, Lemma~\ref{l:fixed points}\eqref{l:fixed points b} implies $\Fix T_A\neq\varnothing$. Since $\gamma\in(0,1)$,  Lemma~\ref{l:T_A ne} implies $T_A$ is $\gamma$-averaged nonexpansive. By applying \cite[Theorem~5.15]{bauschkecombettes}, we therefore deduce that $(\bz^k)$ converges weakly to a point $\bz\in\Fix T_A$ and that $\lim_{k\to\infty}\|\bz^{k+1}-\bz^k\|=0$.

By Lemma~\ref{l:fixed points}\eqref{l:fixed points a}, there exists $\bar{x}\in\Hilbert$ such that $(\bz,\bar{x})\in\Omega$. By nonexpansivity of resolvents and boundedness of $(\bz^k)$, it follows that $(\bx^k)$ is also bounded. Further, \eqref{eq:th T_A} and the fact that $\lim_{k\to\infty}\|\bz^{k+1}-\bz^k\|=0$ implies that
\begin{equation}\label{eq:xki}
 \lim_{k\to\infty}\|x_{i}^k-x_{i-1}^k\|=0\quad\forall i=2,\dots, n.
\end{equation}

Next, using the definition of the resolvent, \eqref{eq:th J_A} implies
\begin{equation*}
\left\{\begin{array}{rll}
A^{-1}_1(z_1^k-x_1^k)&\ni x_1^k  & \\
A_i^{-1}\bigl((z_i^k-x_i^k)-(z_{i-1}^k-x_{i-1}^k)\bigr)  &\ni x_i^k & \forall i\in\integ{2}{n-1} \\
A_n(x_n^k)  &\ni x_1^k-x_n^k-(z_{n-1}^k-x_{n-1}^k) & \\
\end{array}\right.
\end{equation*}
which can be written as the inclusion
\begin{equation}\label{eq:demiclosed2}
\hspace{-0.35cm}\left[\begin{pmatrix}
A_1^{-1}\\ A_2^{-1} \\ \vdots \\ A_{n-1}^{-1} \\ A_n\\
\end{pmatrix} + \begin{pmatrix}
0 & 0 & \dots & 0 & -\Id \\
0 & 0 & \dots & 0 & -\Id \\
\vdots & \vdots & \ddots & \vdots & \vdots \\
0 & 0 & \dots & 0 & -\Id \\
\Id & \Id & \dots & \Id & 0 \\
\end{pmatrix}\right]
\begin{pmatrix}
z_1^k-x_1^k \\
(z_2^k-x_2^k)-(z_{1}^k-x_{1}^k) \\
\vdots \\
(z_{n-1}^k-x_{n-1}^k)-(z_{n-2}^k-x_{n-2}^k) \\
x_n^k \\
\end{pmatrix} \ni
\begin{pmatrix}
x_1^k-x_n^k \\
x_2^k-x_n^k \\
\vdots \\
x_{n-1}^k-x_n^k\\
x_1^k-x_n^k \\
\end{pmatrix}.
\end{equation}
Note that the operator in the inclusion \eqref{eq:demiclosed2} is maximally monotone as the sum of two maximally monotone operators, the latter having full domain \cite[Corollary~24.4(i)]{bauschkecombettes}. Consequently, it is demiclosed \cite[Proposition~20.32]{bauschkecombettes}. That is, its graph is sequentially closed in the weak-strong topology.

Let $\mathbf{w}\in\Hilbert^{n}$ be an arbitrary weak cluster point of the sequence $(\bx^k)$. As a consequence of \eqref{eq:xki}, $\mathbf{w}=(x,\dots,x)$ for some $x\in\Hilbert$. Taking the limit along a subsequence of $(\bx^k)$ which converges weakly to $\mathbf{w}$ in \eqref{eq:demiclosed2}, using demiclosedness and unravelling the resulting expression gives
\begin{equation*}
\left\{\begin{array}{rll}
    A_1(x) &\ni z_1-x \\
    A_i(x) &\ni z_i-z_{i-1} & \forall i\in\integ{2}{n-1} \\
    A_n(x) &\ni x-z_{n-1},
   \end{array}\right. 
\end{equation*}   
which implies $(\bz,x)\in\Omega$. In particular, $\bz\in\Fix T_A$ and $x=J_{A_1}(z_1)\in\zer\left(\sum_{i=1}^nA_i\right)$.

In other words, $\mathbf{w}=(x,\dots,x)\in\Hilbert^n$ with $x:=J_{A_1}(z_1)$ is the unique weak sequential cluster point of the bounded sequence $(\bx^k)$. We therefore deduce that $(\bx^k)$ converges weakly to $\mathbf{w}$, which completes the proof.
\end{proof}

The following corollary, which attains the lower bound given in Theorem~\ref{th:lifting}, is an immediate consequence of Theorem~\ref{th:main}. Thus, in the terminology of Ryu~\cite{ryu2020uniqueness}, the scheme described by Theorem~\ref{th:main} has ``minimal lifting''.

\begin{corollary}
Let $n\geq 2$. There exists a frugal resolvent splitting $(T_A,S_A)$ for $\mathcal{A}_n$ with $(n-1)$-fold lifting. Moreover, there exists no frugal resolvents splitting for $\mathcal{A}_n$ with $d$-fold lifting for $d\leq n-2$.
\end{corollary}

The following remark comments on the difficulties of extending \cite[Theorem~4]{ryu2020uniqueness} to a four operator scheme. Such an extension seems not to be straightforward without additional assumptions.
\begin{remark}[Extensions of Ryu splitting]\label{r:extension to ryu}
Let $A=(A_1,A_2,A_3,A_4)\in\mathcal{A}_4$ and $\gamma\in(0,1)$. Consider the
operator $T\colon\Hilbert^3\to\Hilbert^3$ given by
\begin{equation}\label{eq:4op ryu}
 T(\bz) = \bz + \gamma\begin{pmatrix}
                     x_4-x_1 \\
                     x_4-x_2 \\
                     x_4-x_3 \\
                   \end{pmatrix}\text{~~where~~}\left\{\begin{aligned}
                                                 x_1 &= J_{A_1}(z_1) \\
                                                 x_2 &= J_{A_2}(z_2+x_1) \\
                                                 x_3 &= J_{A_3}(z_3+x_2-x_1) \\
                                                 x_4 &= J_{A_4}(x_1-z_1+x_2-z_2+x_3-z_3), \\
                                                \end{aligned}\right.
\end{equation}
which can be considered as a four operator extension of \eqref{eq:ryu our notation}. By using the definition of the resolvents, it can be shown that for this scheme
 $$ \bz\in\Fix T\implies x_1=x_2=x_3=x_4\in\zer\left(A_1+A_2+A_3+A_4\right). $$
However, its fixed point iteration does not converge for all $A=(A_1,A_2,A_3,A_4)\in\mathcal{A}_4$ with $\zer(A_1+A_2+A_3+A_4)\neq\varnothing$. To see this, let $\Hilbert=\mathbb{R}$ and take $A_1=A_2=A_3=A_4=0$, so that $J_{A_1}=J_{A_2}=J_{A_3}=J_{A_4}=\Id$. In this case, $T(\bz)$ simplifies to
 $$ T(\bz) = \bz + \gamma\begin{pmatrix}
                     z_2 \\
                     0 \\
                     z_1+z_3 \\
                   \end{pmatrix} = P\bz\text{~~where~~}P:=\begin{bmatrix}
                   1 & \gamma & 0 \\
                   0 & 1 & 0 \\
                   1 & 0 & 1+\gamma \\
                   \end{bmatrix}. $$
In particular, if $\bz=(0,0,1)$, then $T^k(\bz)=P^k\bz = (1+\gamma)^k\bz$ which diverges as $k\to\infty$.

Interestingly, the operator described by \eqref{eq:4op ryu} can be shown to be
$\gamma$-averaged nonexpansive if the fourth operator $A_4$ is $1$-strongly
monotone, rather than just monotone. Further, if $A_4$ is $\beta$-strongly monotone
for some $\beta\in(0,1)$, then $\frac{1}{\beta}A_4$ is $1$-strongly monotone. In
this case, \eqref{eq:4op ryu} can instead be applied the operators
$\frac{1}{\beta}A_1,\frac{1}{\beta}A_2, \frac{1}{\beta}A_3$,
$\frac{1}{\beta}A_4$ where we note that
$\zer\left(\sum_{i=1}^4A_i\right)=\zer\left(\sum_{i=1}^4\frac{1}{\beta}A_i\right)$.
\end{remark}

\begin{remark}\label{re:x1-xn}
Since the sequence $(\bx^k)$ is weakly convergent to $(x,\dots,x)$, it follows that  $x_i^k-x_j^k\wto 0$ as $k\to+\infty$ for all $i,j\in\{1,\dots,n\}$. However, \eqref{eq:xki}  shows convergence is actually strong. That is, we have
$$\lim_{k\to\infty}\|x_i^k-x_j^k\|=0\quad \forall i,j\in\{1,\dots,n\}. $$
\end{remark}

\subsection{The Limiting Case under Uniform Monotonicity}

We conclude this section with the following theorem which deals with the limiting case in Theorem~\ref{th:main}. That is, the case where $\gamma=1$.
\begin{theorem}\label{th:main uniform mono}
Let $n\geq 2$, let $A=(A_1,\dots,A_n)\in\mathcal{A}_n$ with $\zer\left(\sum_{i=1}^nA_i\right)\neq\varnothing$, and let $\gamma\in(0,1]$. Further suppose $A_i\colon\Hilbert\setto\Hilbert$ is uniformly monotone with modulus $\phi_i$ for all $i\in\{2,\dots,n\}$. Given $\bz^0\in\Hilbert^{n-1}$, let $(\bz^k)\subseteq\Hilbert^{n-1}$ and $(\bx^k)\subseteq\Hilbert^n$ be the sequences given by \eqref{eq:th T_A} and \eqref{eq:th J_A}, respectively.
Then the following assertions hold.
\begin{enumerate}[(a)]
\item\label{it:uni mono a} The sequence $(\bz^k)$ converges weakly to a point $\bz\in\Fix T_A$.
\item\label{it:uni mono b} The sequence $(\bx^k)$ converges strongly to a point $(x,\dots,x)\in\Hilbert^n$ with $x\in\zer\left(\sum_{i=1}^nA_i\right)$.
\end{enumerate}
\end{theorem}
\begin{proof}
Since $\zer\left(\sum_{i=1}^nA_i\right)\neq\varnothing$, Lemma~\ref{l:fixed points}\eqref{l:fixed points a} implies there exists $(\bar{\bz},x)\in\Omega$ which implies $\bar{\bz}\in\Fix T_A$ and $x\in\zer\left(\sum_{i=1}^nA_j\right)$.  By repeating the proof of Lemma~\ref{l:T_A ne} but using uniform monotonicity in place of monotonicity, we obtain
\begin{equation*}
 \|\bz^{k+1}-\bar{\bz}\|^2 + \frac{1-\gamma}{\gamma}\|\bz^k-\bz^{k+1}\|^2 
 + \frac{1}{\gamma}\bigl\|x_1^k-x_n^k\|^2 + 2\gamma\sum_{i=2}^n\phi_{i}\big(\|x_i^k-x\|) \leq \|\bz^k-\bar{\bz}\|^2\quad\forall k\in\mathbb{N}.
\end{equation*}
Since $\gamma\in(0,1]$ this implies
\begin{equation}\label{eq:PR key}
 \|\bz^{k+1}-\bar{\bz}\|^2  + \frac{1}{\gamma}\bigl\|x_1^k-x_n^k\|^2 + 2\gamma\sum_{i=2}^n\phi_{i}\big(\|x_i^k-x\|) \leq \|\bz^k-\bar{\bz}\|^2\quad\forall k\in\mathbb{N}.
\end{equation}
It follows that $(\|\bz^k-\bar{\bz}\|^2)$ is non-increasing and hence convergent. Taking the limit as $k\to+\infty$ in \eqref{eq:PR key} then gives that $x_1^k-x_n^k\to 0$ and $x_i^k\to x$ for all $i\in\{2,\dots,n\}$. Hence the sequence $(\bx^k)$ converges strongly to $(x,\dots,x)\in\Hilbert^n$, which establishes \eqref{it:uni mono b}.

To establish \eqref{it:uni mono a}, first note that $T_A$ is nonexpansive due to Lemma~\ref{l:T_A ne}. Let $\bz'$ be an arbitrary weak cluster point of the bounded sequence $(\bz^k)$.  Since $\bx^k\to(x,\dots,x)$, \eqref{eq:th T_A} implies that $(\Id-T_A)(\bz^k)\to 0$ which, since $\Id-T_A$ is demiclosed~\cite[Theorem~4.27]{bauschkecombettes},  implies $\bz'\in\Fix T_A$. Thus, applying \cite[Theorem~5.5]{bauschkecombettes}, gives that $(\bz^k)$ converges weakly to a point in $\Fix T_A$ as claimed.
\end{proof}

\begin{remark}
The limiting case (\emph{i.e.,} $\gamma=2$) of Douglas--Rachford
splitting~\eqref{eq:dr gamma}  for a finding a zero in the sum of $n=2$
operators is known as \emph{Peaceman--Rachford
  splitting}~\cite{peaceman1955numerical}. Theorem~\ref{th:main uniform mono}
can be considered as an extension of Peaceman--Rachford splitting for $n\geq 2$
operators, in the sense that it represents the limiting case of our proposed method. For $n\geq 2$, a zero in the sum of $n$ operators can be found using Peaceman--Rachford splitting applied to the product formulation (Example~\ref{ex:psf}) provided that all $n$ operators are uniformly monotone. In contrast, Theorem~\ref{th:main uniform mono} only requires $n-1$ of the operators to be uniformly monotone. 
\end{remark}

\section{Distributed Decentralised Optimisation}\label{s:distributed}
The structure of the update step in Algorithm~\ref{alg} is especially well suited to being performed in a distributed decentralised way, without the need for a ``central coordinator'' to enforce consensus. Specifically, we consider a cycle graph with $n$ nodes labelled $1$ through $n$. Each node in the graph represents a device, and two devices can communicate with one another only if their nodes are adjacent. In our setting, this means that node $i$ can communicate with nodes $i-1$ and $i+1 \pmod n$. For each $i\in\{1,\dots,n\}$, we assume that node $i$ only knows the operator $A_{i}$. In addition, we assign updating of $z_{1},\dots, z_{n-1}$ to nodes $2,\dots, n$, respectively. With this in mind, we give the following protocol (Algorithm~\ref{alg:dist}) for distributed decentralised implementation of Algorithm~\ref{alg}.
\begin{algorithm}[!ht]
\caption{Protocol for distributed decentralised implementation of Algorithm~\ref{alg}. \label{alg:dist}}
\SetKwInOut{Input}{Input}
\Input{Let $\gamma\in(0,1)$. For each $i\in\{2,\dots,n\}$, node $i$ chooses $z^0_i\in\Hilbert$.}
\smallskip
\For{$k=1,2,\dots$}{
1. For each $i\in\{2,\dots, n\}$, node $i$ sends $z_{i-1}^k$ to Node $i-1$\;

\vspace{0.5ex}

2. Node $1$ computes $x_{1}^k = J_{A_{1}}(z_{1}^k)$ and sends $x_1^k$ to nodes $2$ and $n$\;

\vspace{0.5ex}

3. For each $i\in\{2,\dots, n-1\}$, node $i$ computes 
$$x_{i}^k = J_{A_{i}}(z_{i}^k -z_{i-1}^k +x_{i-1}^k),$$ sends $x_i^k$ to node $i+1$, and updates $z_{i-1}^{k+1}=z_{i-1}^k+\gamma (x_{i}^k-x_{i-1}^k)$\; 

\vspace{0.5ex}

4. Node $n$ computes $$x_{n}^k = J_{A_{n}}(x_{1}^k -z_{n-1}^k +x_{n-1}^k)$$ and updates $z_{n-1}^{k+1}=z_{n-1}^k+\gamma (x_{n}^k-x_{n-1}^k)$\;
}
\end{algorithm}

In total, the protocol described in Algorithm~\ref{alg:dist} requires each node to send exactly two messages per iteration; node $i$ sends one message to node $i-1$ and one to node $i+1$ $\pmod n$. As node $i$ is responsible for computing the sequence $(x_i^k)$ (which converges to a solution) in this protocol, each node always keeps track of its own approximate solution. Also note that not all steps in Algorithm~\ref{alg:dist} must be completed in the stated order. For example, after completing Step~3, nodes $2,\dots,n-1$ can commence Step~1 of the next iteration without waiting for node $n$ to complete Step~4.

\subsection{Numerical Example: $\ell_1$-Consensus}
To illustrate the method, consider the simple unconstrained optimisation problem given by
\begin{equation}
  \label{eq:decentr}
  \min_{x\in \R} |x-c_{1}| +\dots + |x-c_{n}|,
\end{equation}
where $c_{1},\dots, c_{n}\in \R$ are given. Through its first order optimality
condition, this problem is equivalent to the monotone inclusion
\begin{equation*}
 \text{find~}x\in\R\text{~such that~}0\in\sum_{i=1}^nA_i(x),\quad \text{~where~}A_i:=\partial|\cdot - c_{i}|.
\end{equation*}
We solve this in the decentralised way described above. Since the resolvent of $A_{i}$ has
an explicit form, Algorithm~\ref{alg} is straightforward to implement (using the
protocol described in Algorithm~\ref{alg:dist}).

Alternatively, we can rewrite~\eqref{eq:decentr} as the constrained minimisation problem given by
\begin{equation}
  \label{eq:decentr2}
  \min_{\bx = (x_{1},\dots, x_{n})\in\mathbb{R}^n} |x_{1}-c_{1}| +\dots + |x_{n}-c_{n}|\quad \text{subject to}\quad L \bx = 0,
\end{equation}
where $L\in \R^{n\times n}$ is the \emph{Laplacian} of the cycle graph, that is $L_{ii} = 2$ for all $i\in \integ{1}{n}$ and
$L_{ij}=-1$ if $|i-j| = 1 \pmod n$ and $0$ otherwise. In words, $L\bx$ amounts
to the exchange of the coordinates of $\bx$  between nodes in the way permitted
by the network topology. Formulation~\eqref{eq:decentr2} can be solved using the
\emph{primal dual hybrid gradient (PDHG)} algorithm
\cite{chambolle2011first,condat2013primal} given by 
\begin{equation}\label{eq:PDHG}
\left\{\begin{aligned}
\bx^{k+1} &= \prox_{\tau f} (\bx^k - \tau L^\ast \by^k)\\
\by^{k+1} &= \by^k + \sigma L(2\bx^{k+1}-\bx^k),
\end{aligned}\right.
\end{equation}
where $f(\bx) := \n{\bx-\mathbf{c}}_{1}$ and
$\tau \sigma \n{L}^{2}\leq 1$. 

However, \eqref{eq:PDHG} turns out not to be the best way to use the PDHG to solve \eqref{eq:decentr2}. Instead, one can first replace the constraint in \eqref{eq:decentr2} with the constraint $\sqrt L \bx = 0$, where $\sqrt L$ is the unique positive
semidefinite matrix such that $(\sqrt L)^2 = L$. Then, after making the change of variables $\bv^k = \sqrt L \by^k$, the PDHG scheme~\eqref{eq:PDHG} with $\sqrt L$ in place of $L$ can be expressed as
\begin{equation}\label{eq:PDHG2}
\left\{\begin{aligned}
\bx^{k+1} &= \prox_{\tau f} (\bx^k - \tau \bv^k)\\
\bv^{k+1} &= \bv^k + \sigma L(2\bx^{k+1}-\bx^k),
\end{aligned}\right.
\end{equation}
where $\tau$ and $\sigma$ are only required to satisfy the weaker inequality $\tau \sigma \|L\| \leq 1$. 

Both Algorithm~\ref{alg} and the PDHG algorithm~\eqref{eq:PDHG2} use
approximately the same number of variables and have approximately the same
amount of computation/communication per iteration.  To compare their
performance, we measure the residuals of the two algorithms: $\frac{1}{\gamma}\n{\bz^{k+1}-\bz^{k}}$ for Algorithm~\ref{alg} and
$\left(\frac{1}{\tau} \n{\bx^{k+1}-\bx^{k}}^{2}+\frac{1}{\sigma}\n{\by^{k+1}-\by^{k}}^{2}-2\lr{\sqrt L(\bx^{k+1}-\bx^k),\by^{k+1}-\by^k}\right)^{1/2}$
for PDHG. The latter residual comes from the interpretation of PDHG as the
proximal point algorithm and after some algebra it can be written simply as
$\left(\frac{1}{\tau} \n{\bx^{k+1}-\bx^{k}}^{2}+ \lr{\bv^{k+1}-\bv^k,\bx^k}\right)^{1/2}$.
Comparison of these two residuals is justified because they corresponds to successive iterations of the underlying fixed point operators.
\begin{figure*}[!htb]
\centering
{\includegraphics[width=1\textwidth]{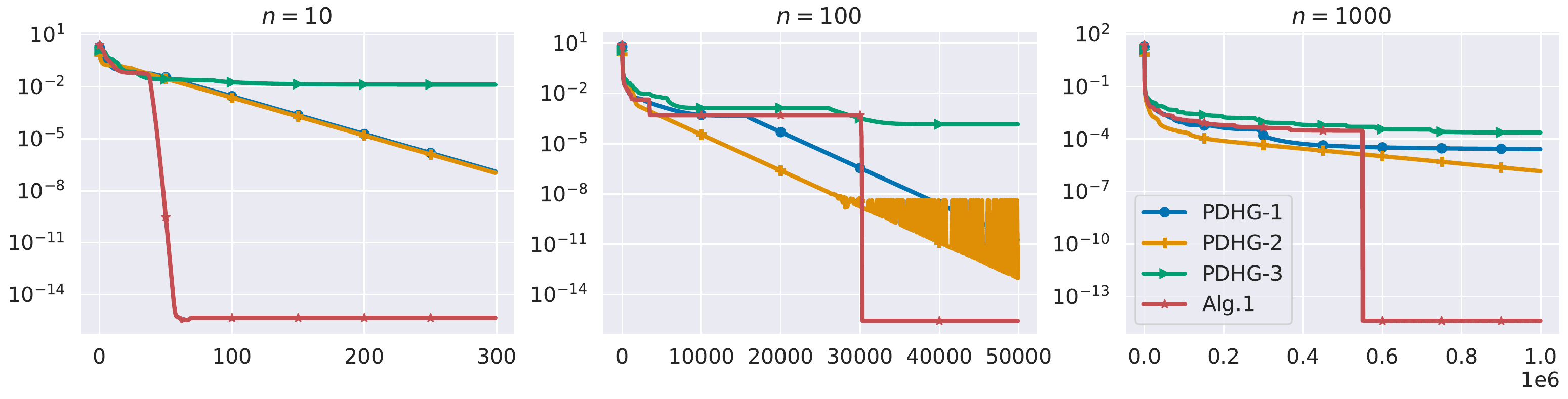}}
\caption{Results for problem~\eqref{eq:decentr}. $x$-axis: iterations, $y$-axis:
  residuals}
\label{fig:l1-norm}
\end{figure*}

For $n\in\{10, 100, 1000\}$, we generated the numbers $c_{i}$ randomly by
sampling the standard normal distribution. The initial point
for Algorithm~\ref{alg} was taken as
$\bz^{0}=\mathbf{0}\in\mathbb{R}^{n-1}$
and for PDHG as $(\bx^{0},\by^{0})=(\mathbf{0},\mathbf{0})\in\mathbb{R}^n\times\mathbb{R}^n$. 
Algorithm~\ref{alg} was run with $\gamma = 0.9$ and PDHG with stepsizes given by
$$(\tau, \sigma) \in  \left\{\Bigl(\frac{1}{\sqrt{\n{L}}}, \frac{1}{\sqrt{\n{L}}}\Bigr), \Bigl(\frac{1}{10\sqrt{\n{L}}}, \frac{10}{\sqrt{\n{L}}}\Bigr), \Bigl(\frac{10}{\sqrt{\n{L}}}, \frac{1}{10\sqrt{\n{L}}}\Bigr)\right\}$$
(we refer to these choices as PDHG versions 1, 2, and 3, respectively). The
results are shown in Figure~\ref{fig:l1-norm} and suggest favourable performance
of Algorithm~\ref{alg}, but further investigation is needed.
\section{Multi-block ADMM}\label{s:admm}
The \emph{alternating direction method of multipliers (ADMM)} is an algorithm for solving minimisation problems of the form
\begin{equation}\label{eq:admm problem}
 \min_{\bw=(w_1,w_2)} f_1(w_1)+f(w_2) \text{~~subject to~~}A_1w_1+A_2w_2 = b,
\end{equation}
where $f_1\colon\Hilbert_1\to(-\infty,+\infty]$ and
$f_2\colon\Hilbert_2\to(-\infty,+\infty]$ are proper lsc convex functions,
$A_1\colon\Hilbert_1\to\Hilbert$ and $A_2\colon\Hilbert_2\to\Hilbert$ are
bounded linear opeartors, and $b\in\Hilbert$ is a given vector. ADMM is known to
be equvialent to applying the Douglas--Rachford method to the Fenchel dual of
\eqref{eq:admm problem}~\cite{gabay1983applications,eckstein1992douglas}.  As the family of resolvent splitting introduced in Section~\ref{s:family} includes the Douglas--Rachford method as the special case when $n=2$ (see Remark~\ref{r:special cases}), it is natural to investigate an analogous equivalence in the setting when $n\geq 3$. 

Before proceeding, we recall some notation. The \emph{Fenchel conjugate} of an extended real-valued function $g$ is denoted $g^*(x):=\sup_{y}\{\langle x,y\rangle-g(y)\}$. Given a set $C$, its \emph{indicator function}, denoted $\iota_C$, is the function which takes the value $0$ on $C$ and $+\infty$ outside $C$.

Consider the separable minimisation problem
\begin{equation}\label{eq:sep convex}
\min_{\bw=(w_1,\dots,w_{n})} \sum_{i=1}^{n}f_i(w_i) \text{~~subject to~~}\sum_{i=1}^{n}A_iw_i = b,
\end{equation}
where $f_1,\dots,f_n$ are proper lsc convex functions, $A_1,\dots,A_n$ are bounded linear operators, and $b\in\Hilbert$ is a given vector. Denote $f(\bw):=\sum_{i=1}^nf_i(w_i)$ and $\bA:=\begin{bmatrix}
A_1 & \dots & A_n \\ \end{bmatrix}$.  Then \eqref{eq:sep convex} can be expressed as the Fenchel primal problem given by
\begin{equation}\label{eq:primal}
 p:=\min_{\bw}\bigl\{ f(\bw) + \iota_{\{b\}}(\bA\bw) \bigr\}.
\end{equation}
which has Fenchel dual given by
$$ d:=\sup_x\bigl\{ -f^*(\bA^*x) - \iota_{\{b\}}^*(-x) \bigr\}. $$
Since $f^*=\bigoplus_{i=1}^nf_i^*$, $\iota_{\{b\}}^*=\langle b,\cdot\rangle$ and $\bA^*=\begin{bmatrix} A_1^* \\ \vdots \\ A_n^* \\ \end{bmatrix}$, this dual problem is equivalent to
\begin{equation}\label{eq:dual}
 \inf_x\bigl\{\sum_{i=1}^{n-1}f_i^*(-A_i^*x)+g^*(x)\bigr\}\text{~~where~~}g^*(x):=f_n^*(-A_n^*x)-\langle b,x\rangle .
\end{equation}

In what follows, we derive a multi-block extension of the ADMM algorithm by applying the new resolvent splitting from Section~\ref{s:family} to the dual problem \eqref{eq:dual}. Such an extension is of interest because, even when $n=3$, the natural ``direct extension of ADMM'' need not converge \cite{chen2016admm}. 

We begin by examining the so-called ``averaged operator'' form of our proposed ADMM extension. To simplify the presentation of our derivation, we will use the following lemmas.

\begin{lemma}\label{l:prox f+g}
Let $h_1\colon\Hilbert'\to(-\infty,+\infty]$ and $h_2\colon\Hilbert\to(-\infty,+\infty]$ be proper lsc convex functions, and let $A\colon\Hilbert'\to\Hilbert$ be a bounded linear operator. Denote $h(z):=h_1^*(-A^*z)+h_2^*(z)$. 
If
\begin{equation}\label{eq:prox f+g}
\begin{aligned}
 (\hat{x},\hat{y})
  &\in \argmin_{x,y}\left\{ h_1(x)+h_2(y)+\langle z, Ax-y\rangle + \frac{1}{2}\|Ax-y\|^2\right\} \\
  &= \argmin_{x,y}\left\{h_1(x)+h_2(y) + \frac{1}{2}\|Ax-y+z\|^2\right\},
\end{aligned}
\end{equation}
then $\prox_{h}(z) = z+(A\hat{x}-\hat{y})$. Moreover, the minimum in \eqref{eq:prox f+g} is attained provided that: (a) $h_1$ is coercive or $A^*A$ is invertible, and (b) $h_2^*=\langle b,\cdot\rangle $ for some vector $b\in\Hilbert'$.
\end{lemma}
\begin{proof}
Suppose the minimum in \eqref{eq:prox f+g} is attained at $(\hat{x},\hat{y})$ and denote $w:=z+(A\hat{x}-\hat{y})$. By first-order optimality in \eqref{eq:prox f+g}, we have $-A^*w\in\partial h_1(\hat{x})$ and $w\in \partial h_2(\hat{y})$. Since $(\partial h_1)^{-1}=\partial h_1^*$ and $(\partial h_2)^{-1}=\partial h_2^*$, it follows that $0 \in (-A)\partial h_1^*(-A^*w)+\partial h_2^*(w) +(w-z)$ which, in turn, implies 
$$ w\in\argmin_u\{h_1^*(-A^*u)+h_1^*(u)+\frac{1}{2}\|u-z\|^2\}=\argmin_u\{h(u)+\frac{1}{2}\|u-z\|^2\}. $$
Thus, $\prox_h(z)=w=z+(A\hat{x}-\hat{y})$, which establishes the first claim.

Further, suppose $h_2^*=\langle b,\cdot\rangle$ for some $b\in\Hilbert$. Then $h_2=h_2^{**}=\iota_{\{b\}}$, and hence  $(\hat{x},\hat{y})$ attains the minimum in \eqref{eq:prox f+g} if and only if $\hat{y}=b$ and
\begin{equation}\label{eq:prox f+g 2}
 \hat{x} \in \argmin_{x}\left\{\phi(x):=h_1(x) + \frac{1}{2}\|Ax-b+z\|^2\right\}. 
\end{equation}
If $h_1$ is coercive, then $\phi$ is also coercive due to nonnegativity $\|\cdot\|^2$. If $A^*A$ is invertible, then $x\mapsto \frac{1}{2}\|Ax-b+z\|^2 = \frac{1}{2}\langle x,A^*Ax\rangle + \langle x,A^*(z-b)\rangle + \frac{1}{2}\|z-b\|^2$
is supercoercive by \cite[Exercise~17.1]{bauschkecombettes} and hence $\phi$ is cocercive by  \cite[Proposition~11.14]{bauschkecombettes}. Thus, in either case, $\phi$ is coercive, and so the minimum in \eqref{eq:prox f+g 2} is attained due to \cite[Proposition~11.15]{bauschkecombettes}. This establishes the second claim and completes the proof.
\end{proof}

\begin{lemma}\label{l:boudned}
Let $h\colon\Hilbert\to(-\infty,+\infty]$ be a proper lsc convex function,  let $A\colon\Hilbert'\to\Hilbert$ be a bounded linear operator, and let $(u^k)\subseteq\Hilbert$ be a bounded sequence. Suppose that $h$ is coercive or $A^*A$ is invertible. Then, any sequence $(w^k)\subseteq\Hilbert'$ satisfying
\begin{equation}\label{eq:l wk}
 w^k \in \argmin_{w}\left\{h(w)+\frac{1}{2}\|Aw+u^k\|^2\right\}\quad\forall k\in\mathbb{N}, 
\end{equation}
is bounded.
\end{lemma}
\begin{proof}
Suppose, by way of a contradiction, that $(w^k)$ is not bounded. Then, without loss of generality, we assume that $\|w^k\|\to+\infty$ as $k\to+\infty$. Using the same argument as in the second part of the proof of Lemma~\ref{l:prox f+g}, we deduce that function $\phi(w):= h(w) + \frac{1}{4}\|Aw\|^2$ is coercive and hence $\phi(w^k)\to+\infty$ as $k\to+\infty$.
Since $h$ is proper, \eqref{eq:l wk} implies 
\begin{align*}
+\infty >  h(w^1)+\frac{1}{2}\|Aw^1+u^k\|^2 
  \geq h(w^k) + \frac{1}{2}\|Aw^k+u^k\|^2 \geq h(w^k) + \frac{1}{4}\|Aw^k\|^2 - \frac{1}{2}\|u^k\|^2 \quad\forall k\in\mathbb{N}.
\end{align*}
Since $(u^k)$ is bounded, this inequality implies that the sequence $\bigl(h(w^k) + \frac{1}{4}\|Aw^k\|^2\bigr)$ is bounded above. This is a contradiction, and so the proof is complete.
\end{proof}

\begin{algorithm}[!htb]
\caption{Multi-block ADMM for \eqref{eq:sep convex} in averaged operator form. \label{alg:admm}}
\SetKwInOut{Input}{Input}
\Input{Choose $\bz^0=(z^0_1,\dots,z^0_{n-1})$ and $\gamma\in(0,1)$.}
\smallskip
\For{$k=1,2,\dots$}{
1.~Compute $\bw^k=(w^k_1,\dots,w^k_n)$ according to
\begin{subequations}\label{eq:multi-ADMM averaged}
\begin{align}[left = \empheqlbrace\,]
w^k_1 &\in \argmin_{w_1}\bigl\{f_1(w_1) + \frac{1}{2}\|A_1w_1+z_1^k\|^2\bigr\}  \label{eq:multi-ADMM averaged a}\\
w^k_i &\in \argmin_{w_i}\bigl\{ f_i(w_i) + \frac{1}{2}\|\sum_{j=1}^{i-1}A_jw_j^k+A_iw_i+z_i^k\|^2\bigr\} \hspace{1.5cm} \forall i\in \integ{2}{n-1} \label{eq:multi-ADMM averaged b} \\
w^k_n &\in \argmin_{w_n}\bigl\{f_n(w_n) + \frac{1}{2}\|2A_1w_1^k+\sum_{j=2}^{n-1}A_jw_j^k+A_nw_n-b+z_1^k\|^2\bigr\} \label{eq:multi-ADMM averaged c} 
\end{align}
2.~Update $\bz^{k+1}=(z^{k+1}_1,\dots,z^{k+1}_{n-1})$ according to
\begin{align}[left = \empheqlbrace\,]
z^{k+1}_i &= z^k_i + \gamma (z^k_{i+1}-z^k_i) + \gamma A_{i+1}w_{i+1}^k \hspace{3.95cm} \forall i\in\integ{1}{n-2}\label{eq:multi-ADMM averaged d} \\
z^{k+1}_{n-1} &= z^k_{n-1} + \gamma (z_1^k-z^k_{n-1}) + \gamma(A_1w_1^k+A_nw_n^k-b). \label{eq:multi-ADMM averaged e}
\end{align}
\end{subequations}
}
\end{algorithm}

Recall that $(\bw,x)$ is a \emph{Kuhn--Tucker pair} for \eqref{eq:primal} if $-\bA^* x\in\partial f(\bw)$ and $x\in\partial \iota_{\{b\}}(\bA\bw)$. In particular, when such a pair exists, $\bw$ is a primal solution and $x$ is a dual solution. Moreover, by using properties of the Fenchel conjugate, it can be seen that $(\bw,x)$ being a Kuhn--Tucker pair is equivalent to $\bw\in\partial f^*(-\bA^*x)$ and $\bA\bw\in\partial\iota_{\{b\}}^*(x)$ which implies
\begin{equation}\label{eq:zer Fis}
\begin{aligned} 0 \in (-\bA)\partial f^*(-\bA^*x) + \partial\iota_{\{b\}}^*(x) 
&= \sum_{i=1}^n(-A_i)\partial f_i^*(-A_i^*x)+ \partial\langle b,\cdot\rangle(x) \\
&\subseteq \sum_{i=1}^{n-1}\partial\bigl(f_i^*\circ(-A_i^*)\bigr)(x) + \partial g^*(x), 
\end{aligned}
\end{equation}

We are ready to derive our multi multi-block extension of ADMM and analyse its convergence. 
The averaged operator form of our proposed algorithm is given in  Algorithm~\ref{alg:admm} and its convergence analysed in Theorem~\ref{th:admm ave}.

\begin{theorem}[Multi-block ADMM -- averaged operator form]\label{th:admm ave}
Let $n\geq 2$. Suppose \eqref{eq:primal} has a Kuhn--Tucker pair, and that, for each $i\in\{1,\dots,n\}$, either $f_i$ is coercive or $A_i^*A_i$ is invertible. Let $\gamma\in(0,1)$. Given $\bz^0=(z^0_1,\dots,z^0_{n-1})$, consider the sequences $(\bz^k)$ and $(\bw^k)$ given by Algorithm~\ref{alg:admm}.
Then:
\begin{enumerate}[(a)]
\item The sequence $(\bw^k)$ is bounded and every weak limit point is a solution of the primal problem~\eqref{eq:primal}. Moreover, if $A_i^*A_i$ is invertible for all $i\in\{1,\dots,n\}$, then $(\bw^k)$ converges strongly.
\item For each $i\in\{1,\dots,n-1\}$, the sequence $(z_i^k+\sum_{j=1}^iA_jw_j^k)$ converges weakly to a solution of the dual problem~\eqref{eq:dual}.
\item The residual sequence converges strongly to zero. That is, $\sum_{i=1}^nA_iw_i^k\to b$ as $k\to+\infty$.
\end{enumerate}
\end{theorem}
\begin{proof}
Let $F:=(F_1,\dots,F_n)\in\mathcal{A}_n$ denote the $n$-tuple of maximally monotone operators given by
\begin{equation}\label{eq:F_i}
 F_i:=\begin{cases}
  \partial\bigl(f_i^*\circ(-A_i^*)\bigr) & i=1,\dots,n-1, \\ 
  \partial g^* & i=n.
  \end{cases} 
\end{equation}
Since \eqref{eq:primal} has a Kuhn--Tucker pair, \eqref{eq:zer Fis} shows that $\zer\bigl(\sum_{i=1}^nF_i\bigr)\neq\varnothing$. Hence $F=(F_1,\dots,F_n)$ satisfies all the assumptions required to apply Theorem~\ref{th:main}. To this end, let $(\bz^k)$ denote the sequence given by 
\begin{equation}\label{eq:zi}
 \bz^{k+1}:=T_F(\bz^k)=\bz^k +\gamma \begin{pmatrix}
x_2^k-x_1^k \\
x_3^k-x_2^k \\
\vdots \\
x_{n}^k-x_{n-1}^k \\
\end{pmatrix},
\end{equation}
where the sequence $(\bx^k)$ is given by
\begin{equation}\label{eq:prox dr}
\left\{\begin{aligned}
x^k_1  &=J_{F_1}(z_1^k) = \prox_{f_1^*\circ(-A_1^*)}(z_1^k),  \\
x^k_i  &= J_{F_i}(z_i^k+x^k_{i-1}-z_{i-1}^k) =\prox_{f_i^*\circ(-A_i^*)}(z_i^k+x^k_{i-1}-z_{i-1}^k)\quad   &\forall i\in\integ{2}{n-1},\\
x^k_n  &= J_{F_n}(x_1^k+x^k_{n-1}-z^k_{n-1}) =\prox_{g^*}(x_1^k+x^k_{n-1}-z^k_{n-1}). 
\end{aligned}\right.
\end{equation}
Using Theorem~\ref{th:main} applied to $F\in\mathcal{A}_n$, we deduce that $(\bz^k)$ converges weakly to a point ${\bz\in\Fix T_F}$ and that $(\bx^k)$ converges weakly to a point $(x,\dots,x)$ with $x\in\zer\bigl(\sum_{i=1}^n F_i\bigr)$. Moreover, Remark~\ref{re:x1-xn} gives  that $\|x^k_i-x^k_j\|\to0$ as $k\to\infty$ for all $i,j\in\{1,\dots,n\}$.

Next, using Lemma~\ref{l:prox f+g}, we express \eqref{eq:prox dr} as
\begin{equation}\label{eq:dr}
\left\{\begin{aligned}
x^k_1 &= z_1^k+A_1w^k_1,  \\
x^k_i &= z_i^k+(x^k_{i-1}-z_{i-1}^k) + A_iw^k_i  &\forall i\in\integ{2}{n-1} \\
x^k_n &= x_1^k+(x^k_{n-1}-z^k_{n-1})+A_nw^k_n-b 
\end{aligned}\right.
\end{equation}
for some sequence $(\bw^k)$ satisfying
\begin{equation}\label{eq:wi}
\left\{\begin{aligned}
w^k_1 &\in \argmin_{w_1}\left\{f_1(w_1) + \frac{1}{2}\|A_1w_1+z_1^k\|^2\right\}  \\
w^k_i &\in \argmin_{w_i}\left\{f_i(w_i) + \frac{1}{2}\|A_iw_i+z_i^k+x^k_{i-1}-z_{i-1}^k\|^2\right\}\quad \forall i\in\integ{2}{n-1}  \\
w^k_n &\in \argmin_{w_n}\left\{f_n(w_n) + \frac{1}{2}\|A_nw_n-b+ x_1^k+x^k_{n-1}-z^k_{n-1}\|^2\right\} .
\end{aligned}\right.
\end{equation}
By substituting \eqref{eq:dr} into \eqref{eq:wi} and \eqref{eq:zi}, we obtain \eqref{eq:multi-ADMM averaged}. Thus, to complete the proof, it remains to establish assertions (a), (b) and (c), which we do in reverse order.
\begin{enumerate}[(a)]
\item[(c)] Summing the system of equations in \eqref{eq:dr} gives
\begin{equation*}
  A_1w_1^k+\dots + A_nw_n^k -  b = x_n^k-x_1^k \to 0\text{~~as~~}k\to+\infty.
\end{equation*}

\item[(b)] Follows by combining weak convergence of $(\bx^k)$ to $(x,\dots,x)$ with \eqref{eq:dr}.

\item[(a)] Since $(\bz^k)$ and $(\bx^k)$ are bounded, Lemma~\ref{l:boudned} implies that $(\bw^k)$ given by \eqref{eq:wi} is also bounded. Let $\bw=(w_1,\dots,w_n)$ be a weak cluster point $(\bw^k)$. Combining \eqref{eq:dr} and \eqref{eq:wi} gives
\begin{equation}\label{eq:admm1}
\partial f_i(w_i^k)\ni -A^*_ix_i^k \quad\forall i\in\{1,\dots,n\}\implies \partial f(\bw^k)  + \bA^*x^k_1 \ni \begin{pmatrix}
 0 \\ A_2^*(x_1^k-x_2^k) \\ \vdots \\ A_n^*(x_1^k-x_n^k) \\ 
 \end{pmatrix}.
\end{equation}
Since $\iota_{\{b\}}^*=\langle b,\cdot\rangle$, we have $(\partial \iota_{\{b\}})^{-1}=b$.  Thus, using the equality in \eqref{eq:admm1}, we deduce that
\begin{equation}\label{eq:admm2}
 (\partial \iota_{\{b\}})^{-1}(x_1^k) = b = \bA\bw^k + x_1^k-x_n^k \implies  (\partial \iota_{\{b\}})^{-1}(x_1^k)-\bA\bw^k \ni x_1^k-x_n^k. 
\end{equation}
By combining \eqref{eq:admm1} and \eqref{eq:admm2}, we obtain the maximally monotone inclusion
\begin{equation*}
 \left[ \binom{\partial f}{(\partial \iota_{\{b\}})^{-1}} + \begin{pmatrix}
0 & A^* \\ -A & 0 \\ \end{pmatrix}\right]\binom{\bw^k}{x_1^k} \ni  \begin{pmatrix}
 0 \\ A_2^*(x_1^k-x_2^k) \\ \vdots \\ A_n^*(x_1^k-x_n^k) \\ x_1^k-x_n^k \\
 \end{pmatrix}. 
\end{equation*}
Since the graph of a maximally monotone operators is demiclosed~\cite[Proposition~20.32]{bauschkecombettes}, taking the limit along a subsequence of $(\bw^k)$ which converges to $\bw$ and unravelling the resulting expression gives 
$$ -\bA^*x\in \partial f(\bw)\text{~~and~~}\bA\bw(\partial\iota_{\{b\}})^{-1}(x)\iff x\in\partial\iota_{\{b\}}(\bA\bw).$$
That is, $(\bw,x)$ is a Kuhn-Tucker pair for \eqref{eq:primal} and so, in particular, $\bw$ solves the primal problem~\eqref{eq:primal}. Moreover, if $A_i^*A_i$ are invertible for all $i\in\{1,\dots,n\}$, then $\bA^*\bA$ is also invertible \cite[Example~3.29]{bauschkecombettes}. Consequently, $\bA\bw^k\to b$ implies that $\bw^k=(\bA^*\bA)^{-1}\bA^*\bA\bw^k\to(\bA^*\bA)^{-1}\bA^*b$. 
\end{enumerate}
The proof is now complete.
\end{proof}

\begin{remark}
The form of ADMM in Theorem~\ref{th:admm ave} is a a multi-block version of the ``averaged operator'' form of two-block ADMM as it appears in \cite[Equations~(77)--(79)]{giselsson2016line}. In the two-block setting, Theorem~\ref{th:admm ave} includes \cite[Proposition~5.4.1]{bertsekas2015convex} as a special case.
\end{remark}

Next, we express the multi-block ADMM extension in Theorem~\ref{th:admm ave} in terms of the \emph{augmented Lagrangian}. Recall that the \emph{augmented Lagrangian} of \eqref{eq:primal} is the function $\mathcal{L}$ given by
$$ \mathcal{L}(w_1,\dots,w_n,\mu) := \sum_{i=1}^nf(w_i)+\bigl\langle\mu,\sum_{i=1}^nA_iw_i-b\bigr\rangle + \frac{1}{2}\bigl\|\sum_{i=1}^nA_iw_i-b\bigr\|^2. $$
The augmented Lagrangian form of our proposed algorithm in terms of $\mathcal{L}$ is given in Algorithm~\ref{alg:admm2} and its converged analysed in Corollary~\ref{cor:admm aug}.
\begin{algorithm}[!h]
\caption{Multi-block ADMM for \eqref{eq:sep convex} in augmented Lagrangian form. \label{alg:admm2}}
\SetKwInOut{Input}{Input}
\Input{Choose $\bmu^0=(\mu^0_1,\dots,\mu^0_{n})$ and $\gamma\in(0,1)$.}
\smallskip
\For{$k=1,2,\dots$}{
1.~Compute $\bw^k=(w^k_1,\dots,w^k_n)$ according to
\begin{subequations}\label{eq:multi-ADMM augmented}
\begin{align}[left = \empheqlbrace\,]
w^{k+1}_1 &=\argmin_{w_1}\,\mathcal{L}(w_1,w_2^{k},\dots,w_n^{k},\mu_1^k) \label{eq:multi-ADMM augmented a} \\
w^{k+1}_i &=\argmin_{w_i}\,\mathcal{L}(w_1^{k+1},\dots,w_i^{k+1},w_i,w_{i+1}^{k}\dots,w_n^{k-1},\mu_i^k) \quad \forall i\in\integ{2}{n-1}\rrbracket \label{eq:multi-ADMM augmented b} \\
w^{k+1}_n &= \argmin_{w_n}\,\mathcal{L}(w_1^{k+1},\dots,w_{n-1}^{k+1},w_n,\mu_n^k) \label{eq:multi-ADMM augmented c} 
\end{align}
2.~Update $\bmu^{k+1}=(\mu^{k+1}_1,\dots,\mu^{k+1}_{n-1})$ according to
\begin{align}[left = \empheqlbrace\,]
\begin{split}
\mu^{k+1}_i &=\mu^{k}_{i+1}+\sum_{j=i+2}^n\bigl(A_jw_j^{k}-A_jw_j^{k+1}\bigr) \\
&\qquad\qquad + (1-\gamma)\bigl(\mu^k_i-\mu_{i+1}^k+A_{i+1}w_{i+1}^{k}-A_{i+1}w_{i+1}^{k+1}\bigr)
\end{split} 
\forall i\in\integ{1}{n-1} \label{eq:multi-ADMM augmented d}  \\
\mu^{k+1}_n &= \mu^{k+1}_1 + \left(\sum_{j=1}^nA_jw_j^{k+1}-b\right) \label{eq:multi-ADMM augmented e}  
\end{align}
\end{subequations}
}
\end{algorithm}

\begin{corollary}[Multi-block ADMM -- augmented Lagrangian form]\label{cor:admm aug}
Let $n\geq 2$.  Suppose \eqref{eq:primal} has a Kuhn--Tucker pair, and that, for each $i\in\{1,\dots,n\}$, either $f_i$ is coercive or $A_i^*A_i$ is invertible. Let $\gamma\in(0,1)$. Given $\bmu^0=(\mu^0_1,\dots,\mu^0_{n})$, consider the sequences $(\bw^k)$ and  $(\bmu^k)$ given by Algorithm~\ref{alg:admm2}.
Then:
\begin{enumerate}[(a)]
\item The sequence $(\bw^k)$ is bounded and every weak limit point is a solution of the primal problem~\eqref{eq:primal}. Moreover, if $A_i^*A_i$ is invertible for all $i\in\{1,\dots,n\}$, then $(\bw^k)$ converges strongly.
\item The sequences $(\mu_1^k),\dots,(\mu_n^k)$ converge weakly to a solution of the dual problem~\eqref{eq:dual}.
\item The residual sequence converges strongly to zero. That is, $\sum_{i=1}^nA_iw_i^k\to b$ as $k\to+\infty$.
\end{enumerate}
\end{corollary}
\begin{proof}
The proof is a continuation of the proof of Theorem~\ref{th:admm ave}. To this end, let $(\bz^k)$, $(\bx^k)$ and $(\bw^k)$ be as in the proof of Theorem~\ref{th:admm ave}. Define $\bmu^k=(\mu^k_1,\dots,\mu^k_n)$ according to
\begin{equation}\label{eq:z mu}
 \mu^k_i = \begin{cases}
 z^{k+1}_i -  \bigl(\displaystyle\sum_{j=i+1}^nA_jw_j^{k}-b\bigr) & i=1,\dots,n-1, \\
 z^{k+1}_1 + A_1w_1^{k+1} & i=n.
 \end{cases}
\end{equation}
Substituting \eqref{eq:z mu} into \eqref{eq:multi-ADMM averaged a}-\eqref{eq:multi-ADMM averaged c} gives
\begin{equation*}\left\{\begin{aligned}
w^{k+1}_1 &=\argmin_{w_1}\bigl\{f_1(w_1) + \frac{1}{2}\|A_1w_1+\sum_{j=2}^nA_jw_j^{k}-b+\mu_1^{k}\|^2\bigr\}  \\
w^{k+1}_i &=\argmin_{w_i}\bigl\{f_i(w_i) + \frac{1}{2}\|\sum_{j=1}^{i-1}A_jw_j^{k+1}+A_iw_i+\sum_{j=i+1}^nA_jw_j^{k}-b+\mu_i^{k}\|^2\bigr\} \quad \forall i\in\llbracket 2,n-1\rrbracket  \\
w^{k+1}_n &= \argmin_{w_n}\bigl\{f_n(w_n) + \frac{1}{2}\|\sum_{j=1}^{n-1}A_jw_j^{k+1}+A_nw_n-b+\mu^k_n\|^2\bigr\},
\end{aligned}\right.
\end{equation*}
which is equivalent to \eqref{eq:multi-ADMM augmented a}-\eqref{eq:multi-ADMM augmented c}. Similarly, substituting \eqref{eq:z mu} into \eqref{eq:multi-ADMM averaged d}~\&~\eqref{eq:multi-ADMM averaged e} gives \eqref{eq:multi-ADMM augmented d}~\&~\eqref{eq:multi-ADMM augmented e}.

We therefore have that assertions (a) and (c) follow immediately from Theorem~\ref{th:admm ave}(a)~\&~(c). To establish assertion (b), observe that
\begin{equation}\label{eq:mu alt}
\mu_i^k =
\begin{cases}
 \bigl(z^{k+1}_i+\displaystyle\sum_{j=1}^iA_jw_j^{k+1}\bigr)  + \displaystyle\sum_{j=1}^i\left(A_jw_j^{k}-A_jw_j^{k+1}\right)-  \bigl(\displaystyle\sum_{j=1}^nA_jw_j^{k}-b\bigr)  & i\in\llbracket 1,n-1\rrbracket\\
 z^{k+1}_1 + A_1w_1^{k+1} &i=n.
\end{cases}\end{equation}
Convergence of $(\mu^k_i)$ to a dual solution of \eqref{eq:dual} follows with the help of Theorem~\ref{th:admm ave}(b)~\&~(c) by taking the limit as $k\to+\infty$ in \eqref{eq:mu alt}. In taking this limit, we note that $\|A_jw_j^{k}-A_jw_j^{k+1}\|\to 0$ for all $j\in\{1,\dots,n\}$ as a consequence of combining \eqref{eq:dr} with the fact that  $\|\bz^{k+1}-\bz^k\|\to 0$ and that $\|x_{i}^k-x_j^k\|\to 0$ for all $i,j\in\{1,\dots,n\}$.
\end{proof}

\begin{remark}
If the operators $F_2,\dots,F_n$ in \eqref{eq:F_i} are uniformly monotone, then
Theorem~\ref{th:main uniform mono} can be used in place of Theorem~\ref{th:main}
to allow the limiting case with $\gamma=1$. In this case, setting $i=n-1$ in \eqref{eq:multi-ADMM augmented d} gives $\mu^{k}_{n-1}=\mu^k_n$ for all $k\in\mathbb{N}$, which can be substituted into \eqref{eq:multi-ADMM augmented e} to eliminate the sequence $(\mu^k_n)$ from Algorithm~\ref{alg:admm}. In particular, in the case when $\gamma=1$ and $n=2$, \eqref{eq:multi-ADMM augmented} reduces to the standard form of two-block ADMM given by
\begin{equation*}\left\{\begin{aligned}
w^{k+1}_1 &=\argmin_{w_1}\,\mathcal{L}(w_1,w_2^{k},\mu_1^k) \\
w^{k+1}_2 &= \argmin_{w_2}\,\mathcal{L}(w_1^{k+1},w_2,\mu_1^k) \\
\mu^{k+1}_1 &= \mu^{k+1}_1 + \left(A_1w_1^{k+1}+A_2w_2^{k+1}-b\right).
\end{aligned}\right.
\end{equation*}
\end{remark}

\subsection{Numerical Example: Robust PCA with Partial Information}
In this section, we provide a numerical illustration of our multi-block ADMM method in robust principle component analysis. Given a matrix $M\in\mathbb{R}^{m\times n}$, the \emph{robust principle component analysis (PCA)} problem is to find two matrices $L,S\in\mathbb{R}^{m\times n}$ such that $L$ is low rank and $S$ is sparse. In~\cite{wright2009robust}, this problem is formulated as the convex minimisation problem
\begin{equation}\label{eq:robust pca}
\min_{L,S\in\mathbb{R}^{m\times n}}\|L\|_\ast+\lambda\|S\|_1\quad\text{subject to}\quad L+S=M.
\end{equation}
where $\|\cdot\|_\ast$ denotes the nuclear norm, $\|\cdot\|_1$ denotes the
$\ell_1$-norm, and $\lambda>0$. The parameter $\lambda$ is chosen to balance the
competing effects of the nuclear norm, which promotes low rankedness of $L$, and the $\ell_1$-norm, which promotes sparsity in $S$.

In practice, it is not always possible to observe all the entries of the matrix $M\in\mathbb{R}^{m\times n}$. Instead, one observes only entries of $M$ corresponding to some index set $\Omega\subseteq\{1,\dots,m\}\times\{1,\dots,n\}$. That is, the entry $M_{ij}$ is known if $(i,j)\in\Omega$.  To deal with robust PCA problem having only partial knowledge of $M$, \cite{tao2011recovering} proposed an extension of \eqref{eq:robust pca} given by
\begin{equation*}
\min_{L,S,P\in\mathbb{R}^{m\times n}}\|L\|_\ast+\lambda\|S\|_1\quad\text{subject to}\quad \|P_{\Omega}(M-L-S)\|_{F}\leq \delta,
\end{equation*}
where $P_{\Omega}$ denotes the orthogonal projection onto matrices supported by $\Omega$ and $\delta>0$. By setting $C:=\{D\in\mathbb{R}^{m\times n}:\|P_{\Omega}(D)\|_F\leq \delta\}$, this problem can be expressed as given by
\begin{equation}\label{eq:robust pca partial}
  \min_{L,S,D\in\mathbb{R}^{m\times n}}\|L\|_\ast+\lambda\|S\|_1+\iota_{C}(D)\quad\text{subject to}\quad L+S+D=M.
\end{equation}
The problem \eqref{eq:robust pca partial} can be solved using a three block ADMM-type method known as \emph{alternating splitting augmented Lagrangian method (ASALM)} \cite{tao2011recovering}. This method is given by
\begin{equation}\label{eq:tao yuan}
\left\{\begin{aligned}
 D^{k+1} &= \displaystyle\argmin_{D\in C}\bigl\{\frac{1}{2}\|L^k+S^k+D-M+\mu^k\|^2\bigr\} \\
 S^{k+1} &= \displaystyle\argmin_{S\in\mathbb{R}^{m\times n}}\bigl\{\lambda\|S\|_1 + \frac{1}{2}\|L^k+S+D^{k+1}-M+\mu^k\|^2\bigr\} \\
 L^{k+1} &= \displaystyle\argmin_{L\in\mathbb{R}^{m\times n}}\bigl\{\|L\|_\ast + \frac{1}{2}\|L+S^{k+1}+D^{k+1}-M+\mu^k\|^2\bigr\} \\
 \mu^{k+1} &= \mu^k + \bigl(L^{k+1}+S^{k+1}+D^{k+1}-M\bigr). 
\end{aligned}\right.
\end{equation}
Despite its numerical performance, the theoretical underpinnings of ASALM  are not completely understood. Indeed, the authors of \cite{tao2011recovering} were only able to showed boundedness of the ASALM iterates in a special case. This is consistent with counter-examples for convergence of the direct extension of ADMM to three blocks \cite{chen2016admm}.

\begin{remark}
Noting that \eqref{eq:robust pca partial} is of the form specified by \eqref{eq:sep convex}, we may apply the three block version of our multi-block ADMM developed in Section~\ref{s:admm}. It is interesting to compare the similarities of \eqref{eq:tao yuan} with \eqref{eq:multi-ADMM augmented} in limiting case with $\gamma=1$ which can be expressed as
\begin{equation*}
\left\{\begin{aligned}
 D^{k+1} &= \displaystyle\argmin_{D\in C}\bigl\{\frac{1}{2}\|L^k+S^k+D-M+\mu^k_1\|^2\bigr\} \\
 S^{k+1} &= \displaystyle\argmin_{S\in\mathbb{R}^{m\times n}}\bigl\{\lambda\|S\|_1 + \frac{1}{2}\|L^k+S+D^{k+1}-M+\mu^k_2\|^2\bigr\} \\
 L^{k+1} &= \displaystyle\argmin_{L\in\mathbb{R}^{m\times n}}\bigl\{\|L\|_\ast + \frac{1}{2}\|L+S^{k+1}+D^{k+1}-M+\mu^k_2\|^2\bigr\} \\
\mu^{k+1}_1 &=\mu^{k}_{2}+\bigl(L^{k}-L^{k+1}\bigr)  \\
\mu^{k+1}_2 &= \mu^{k+1}_1 +\bigl(L^{k+1}+S^{k+1}+D^{k+1}-M\bigr). \\
\end{aligned}\right.
\end{equation*}
As compared to \eqref{eq:tao yuan}, \eqref{eq:multi-ADMM augmented} has two copies of the the ``dual variable'' in its update. The $\mu^k_1$-update in \eqref{eq:multi-ADMM augmented} allows for correction by $L^k-L^{k+1}$ not present in  \eqref{eq:tao yuan}.
\end{remark}

We now compare \eqref{eq:tao yuan} and the multi-block ADMM in Algorithm~\ref{alg:admm} on randomly generated test problems. We took $L$ to be a binary ``checker-board'' matrix, generated $S$ by sampling the standard normal distribution in approximately 15\% of its entries, and generated the set $\Omega$ to be approximately 40\% of entries. The partial matrix $M$ was then computed elementwise according to $M_{ij}=L_{ij}+S_{ij}$ for all $(i,j)\in\Omega$. A random examples of $L,M$ and $\Omega$ generated in this way is shown in Figure~\ref{fig:pca_problem}. 

We ran both methods for a $20\times 20$ and a $40\times 40$ instance of \eqref{eq:robust pca partial}. The best parameters in \eqref{eq:robust pca partial} for this instance were found by trial and error to be $\lambda=0.25$ and $\delta=0.1$. For Algorithm~\ref{alg:admm}, the stepsize $\gamma$ was chosen as $\gamma=0.8$. All methods were initialised with the zero matrices. After $2000$ iterations, the reconstructions of $L$ for both algorithms were inspected and were visually indistinguishable.

\begin{figure*}[!htb]
  \centering
  \includegraphics[width=\textwidth]{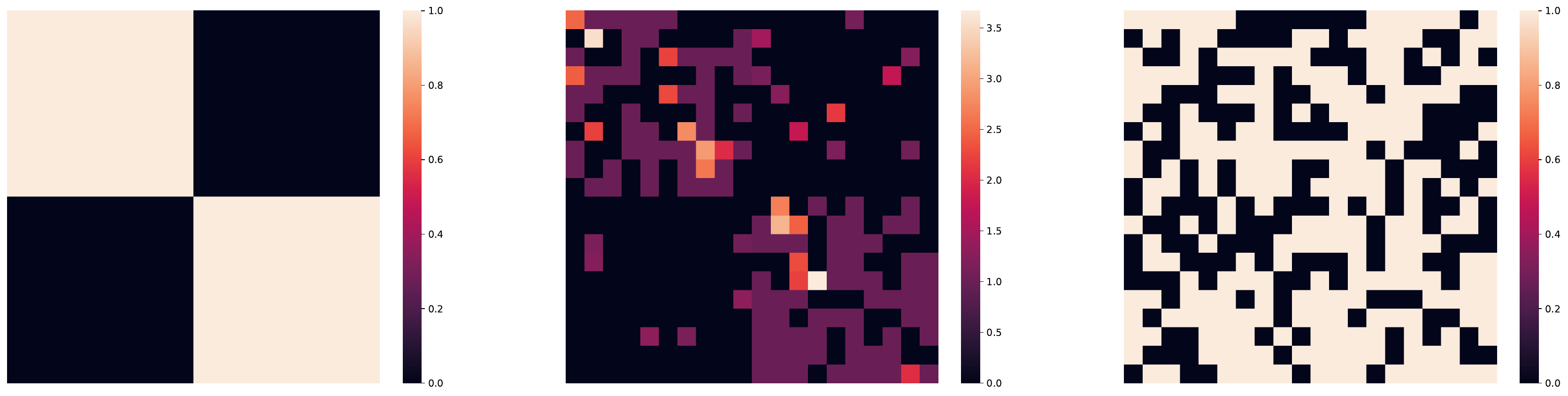}
  \caption{A random instance of~\eqref{eq:robust pca} for $m=n=20$: (left) the low rank
    matrix $L$, (center) the noisy partial matrix $M$, and (right) the set $\Omega$ representing the observed entries. \label{fig:pca_problem}}
\end{figure*}

\begin{figure*}[!htb]
\centering
\begin{subfigure}[t]{0.48\textwidth}
    {\includegraphics[trim={3mm 0 3mm 0},clip,width=1\textwidth]{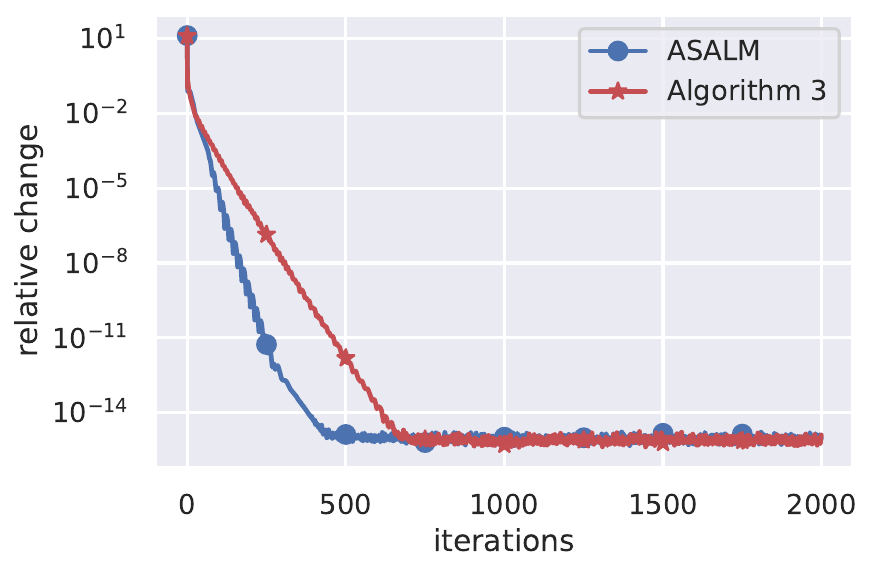}}
\caption{$m=n=20$}
\end{subfigure}
\hfill
\begin{subfigure}[t]{0.48\textwidth}
{\includegraphics[trim={3mm 0 3mm 0},clip,width=1\textwidth]{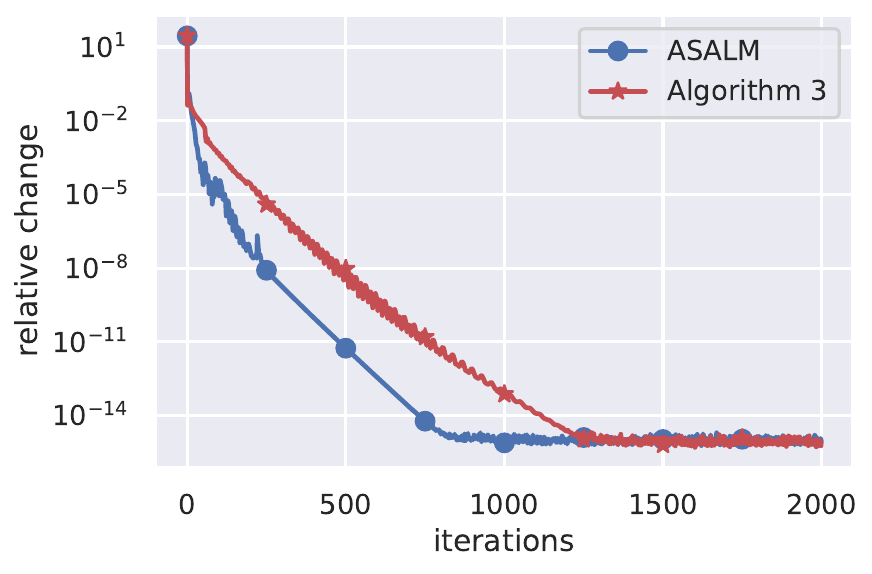}}
\caption{$m=n=40$}
\end{subfigure}
\caption{Results for problem~\eqref{eq:robust pca}. \label{fig:pca-residual}}
\end{figure*}

We compared the progress of both algorithms as a function of iterations in Figure~\ref{fig:pca-residual}. To do so, we monitored the relative change of the recovered low rank and sparse components of the solution which is given by
\begin{equation*}
 \frac{\|(L^{k+1},S^{k+1})-(L^k,S^k)\|}{\|(L^k,S^k)\|+1}.
\end{equation*}
This quantity was used in \cite[Section~9]{tao2011recovering} as a stopping criteria for ASALM. Although the relative change for both methods decays to the same final value, Figure~\ref{fig:pca-residual} shows that ASALM is faster in terms of iterations. However, its poorer theoretical properties, as compared to Algorithm~\ref{alg:admm}, may counterbalance this is in some situations.

\section{Conclusions}\label{s:conclusion}
In this work, we have introduced a family of frugal resolvent splittings with minimial lifting in the sense of Ryu~\cite{ryu2020uniqueness}. We investigated applications of this family in distributed decentralised optimisation and in multi-block ADMM. To conclude, we outline possible directions for further research arising from this work.

\paragraph{Characterising frugal resolvent splittings for $n$ operators.} The proximal point algorithm and the Douglas--Rachford algorithm are the unique frugal resolvent splittings for $n=1$ and $n=2$, respectively \cite{ryu2020uniqueness}. For $n\geq 3$, there seems to be multiple distinct schemes (\emph{i.e.,} this work and \cite{ryu2020uniqueness,aragon2020strengthened}). It would be interesting to characterise and enumerate all possible frugal resolvents splittings for a given number of monotone operators. Iterations with different structure will be potentially useful for distributed decentralised optimisation with non-cyclic network topologies.

\paragraph{Behaviour on infeasible and pathological problems.} In this work, we only analysed our frugal resolvent splitting and our multi-block ADMM in the consistent, non-pathological setting. In the literature, the behaviour of the Douglas--Rachford method and two-block ADMM applied to infeasible and/or pathological problems is relatively well understood \cite{ryu2019douglas} within the framework of Pazy's trichotomy theorem \cite{pazy1971asymptotic}. It would be interesting to analyse the behaviour our methods in potentially infeasible and/or pathological settings.

\paragraph{Iteration complexity of the multi-block ADMM.}
In our analysis of ADMM, Theorem~\ref{th:admm ave} focused on convergence of the iterates generated by the algorithm, but did not consider iteration complexity. The worst-case iteration complexity of two-block ADMM is known to be $O(1/k)$ in the ergodic sense \cite{he2021convergence,monterio2013iteration}. It would be interesting to investigate the iteration complexity of our multi-block ADMM extension to see if it is still $O(1/k)$.

\paragraph{Interpretations in terms of the PPA.}
The Douglas--Rachford algorithm can be interpreted as a proximal point algorithm
applied to the so-called ``splitting operator'' (see
\cite{eckstein1992douglas}). Since our proposed framework is a generalisation of the
Douglas-Rachford algorithm to $n$ operators, it is natural to ask if it can also be understood as an instance of the
proximal point algorithm.

\section*{Acknowledgements}
The work of YM was supported by the Wallenberg
Al, Autonomous Systems and Software Program (WASP) funded by the Knut
and Alice Wallenberg Foundation.  The project number is 305286.
MKT is supported in part by Australian Research Council grant DE200100063.
The authors would like to thank the anoymous referees for helpful comments, which included the improved PDHG formulation given in \eqref{eq:PDHG2}.

\end{document}